\documentclass{amsart}
\usepackage{latexsym,color,amsmath,amsthm,amssymb,amscd,amsfonts,graphicx}

\newcommand{\wt}{\widetilde}
\newcommand{\R}{\mathbb R}
\newcommand{\D}{\mathrm D}
\newcommand{\w}{\mathrm w}
\newcommand{\A}{\mathrm A}
\newcommand{\K}{\mathcal K}

\newcommand{\conv}{\mathrm{conv}}
\newcommand{\bd}{\mathrm{bd}}

\newcommand{\lin}{\mathrm{span}}

\theoremstyle{plain}
\newtheorem{theorem}{Theorem}[section]
\newtheorem{lemma}[theorem]{Lemma}
\newtheorem{corollary}[theorem]{Corollary}
\newtheorem{proposition}[theorem]{Proposition}
\newtheorem{conjecture}[theorem]{Conjecture}

\theoremstyle{definition}

\theoremstyle{remark}
\newtheorem{remark}[theorem]{Remark}
\begin{document}

%%%%%%%%%%%%%%%%%%%%%%%%%%%%%%%%%%%%%%%%%%%%%5

\title[On the isodiametric and isominwidth inequalities for planar bisections]{On the isodiametric and isominwidth inequalities for planar bisections}

    \author[A. Ca\~nete]{Antonio Ca\~nete}
     \address{Departamento de Matem\'atica Aplicada I, Universidad de Sevilla}
     \email{antonioc@us.es}
	
	\author[B. Gonz\'alez Merino]{Bernardo Gonz\'{a}lez Merino}
	\address{Departamento de An\'alisis Matem\'atico, %Facultad de Matem\'aticas,
Universidad de Sevilla}%, Apdo. 1160, 41080-Sevilla, Spain}
	\email{bgonzalez4@us.es}

\thanks{The first author was partially supported by the MICINN project MTM2017-84851-C2-1-P,
and by Junta de Andaluc\'ia grant FQM-325. The research of the second author is a result of the activity developed within the framework
of the Programme in Support of Excellence Groups of the Regi\'on de Murcia, Spain, by Fundaci\'on S\'eneca,
Science and Technology Agency of the Regi\'on de Murcia. The second author was partially supported
by Fundaci\'{o}n S\'{e}neca project 19901/GERM/15 and by MICINN project PGC2018-094215-B-I00, Spain.
}



%\date{\today}
\maketitle

\begin{abstract}
For a given planar convex compact set $K$, consider a bisection $(A,B)$ of $K$
(i.e., $A\cup B=K$ and whose common boundary $A\cap B$ is an injective continuous curve connecting two boundary points of $K$)
minimizing the corresponding maximum diameter (or maximum width) of the regions among all such bisections of $K$.

In this note we study some properties of these minimizing bisections and
we provide analogous to the isodiametric (Bieberbach, 1915), the isominwidth (P\'al, 1921),
the reverse isodiametric (Behrend, 1937),
and the reverse isominwidth (Gonz\'alez Merino \& Schymura, 2018) inequalities.
\end{abstract}

\section{Introduction}

The siblings Alice and Bob are deeply sad due to the loss of their uncle Charlie, who recently passed away.
Soon, they will be awarded with his heritage consisting of a countryside piece of ground.
They have to divide this terrain into two connected pieces of ground, which
must be \emph{equal} according to some \emph{even rule} or \emph{fairness}. In this paper, we
will try to solve their issues, when the rule is either that the diameter or the minimum width
of each of the pieces of ground is as small as possible
(and so, the largest distance in the two pieces is minimized,
or the eventual use of an agrarian harvester is optimized).

Let $\mathcal K^2$ be the family of planar convex bodies
(recall that, as usual, a convex body is a convex compact set)
with non-empty interior.
Throughout this paper, for a given compact set $A\subset\R^2$,
we will denote its \emph{area} (or 2-dimensional Lebesgue measure) by $\A(A)$,
its \emph{diameter} (largest Euclidean distance between two points in $A$) by $\D(A)$,
and its \emph{(minimum) width} (shortest distance between two parallel lines
containing $A$ between them) by $\w(A)$.

For a given $K\in\mathcal K^2$, a \emph{bisection} of $K$ will be
any pair of closed sets $(K_1,K_2)$
satisfying that
\begin{itemize}
\item[(i)] $K=K_1\cup K_2$,
\item[(ii)] $K_1\cap K_2=l([-1,1])$, where $l:[-1,1]\rightarrow K$ is an injective and continuous curve and whose endpoints
$l(-1)$, $l(1)$ are the only points of the curve in the \emph{boundary} $\mathrm{bd}(K)$ of $K$.
\end{itemize}

%\begin{remark}
%We note that the injectivity of the curve $l$ in the previous definition of bisection implies that
%its two endpoints are necessarily distinct.
%\end{remark}

\begin{figure}[ht]
  \includegraphics[width=0.68\textwidth]{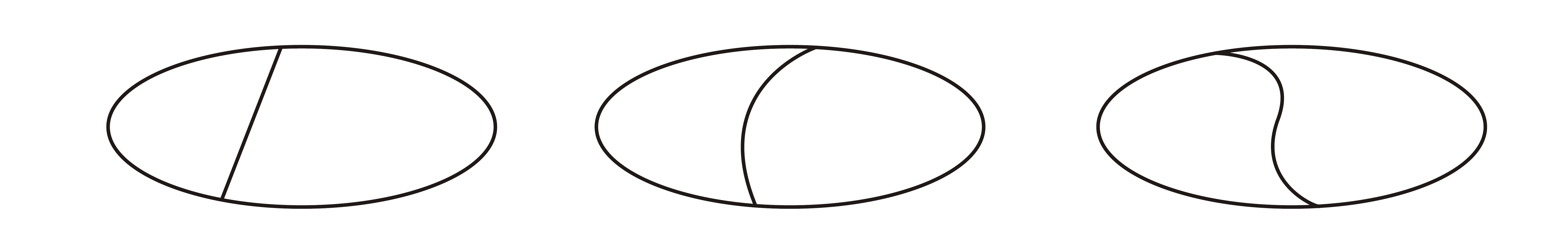}\\
  \caption{Some bisections for the ellipse}\label{fig:bisections}
\end{figure}

For $K\in\mathcal K^2$, let $\mathcal{B}(K)$ be the set of all the bisections of $K$.
We will denote the \emph{infimum of the maximum bisecting diameter} of $K$ by
\begin{equation}\label{eq:db}
\D_B(K):=\inf_{(K_1, K_2)\in\mathcal{B}(K)}\max\{\D(K_1),\D(K_2)\}.
\end{equation}
In some sense, $\D_B(K)$ can be understood, for each $K\in\mathcal K^2$,
as the optimal value for the diameter functional when considering bisections of $K$.
We will see in Lemma~\ref{lem:minimum} that such an infimum is in fact a \emph{minimum}.
We will study in this work the bisections of $K$ which provide $\D_B(K)$, which will be called \emph{minimizing bisections} of $K$,
obtaining also an isodiametric-type inequality relating $\D_B(K)$ and $\A(K)$. %, and even determining the sets attaining the optimal value for \eqref{eq:db}.

Our motivation mainly emanates from a paper by Miori et al \cite{MPS}.
That paper focuses on bisections into two regions of \emph{equal area} minimizing
the maximum bisecting diameter in the setting of \emph{centrally symmetric} planar convex bodies.
Among other results, they prove that for every set in this family, there always exists
a minimizing bisection determined by a line segment \cite[Prop.~4]{MPS}, and describe in \cite[Th.~5]{MPS}
the optimal set for this problem (that is, the set of fixed area with the minimum possible value for the maximum bisecting diameter).
Moreover, for general planar convex bodies they also demonstrate that the minimum value for that functional
when considering bisections by line segments is attained by a centrally symmetric set \cite[Th.~6]{MPS}.
Then, Proposition~\ref{prop:MioriPeriSegura} below follows from these results
(although it is not explicitly stated in \cite{MPS}):
for a given $K\in \mathcal{K}^2$, consider
\[
\widetilde{\D}_{B}(K)=\inf_{(K_1,K_2)\in \mathcal{\widetilde{B}}(K)}\max\{\D(K_1),\D(K_2)\},
\]
where
\[
\mathcal{\widetilde{B}}(K)=\{(K_1,K_2)\in\mathcal{B}(K): K_1\cap K_2\text{ is a line segment, }\A(K_1)=\A(K_2)\}.
\]
%Notice that $\mathcal{\widetilde{B}}(K)$ contains the bisections of $K$ determined by a line segment providing two equal-area regions.
In~\cite[Ex.~2.3]{MPS} the authors consider the set
\[
Q=\left\{(x_1,x_2)\in\R^2:-\frac1{\sqrt{5}}\leq x_1\leq\frac1{\sqrt{5}}\,\,\text{ and }\,\,\left(x_1\pm\frac1{\sqrt{5}}\right)^2+x_2^2\leq 1\right\},
\]
proving that the only bisection of $Q$ in $\widetilde{B}(Q)$ providing the value $\wt{\D}_B(Q)$ is the bisection $(Q^+,Q^-)$,
where
$Q^+=Q\cap\{(x_1,x_2)\in\R^2:x_2\geq 0\}$ and $Q^-=Q\cap\{(x_1,x_2)\in\R^2:x_2\leq 0\}$.

\begin{figure}[ht]
  \includegraphics[width=0.65\textwidth]{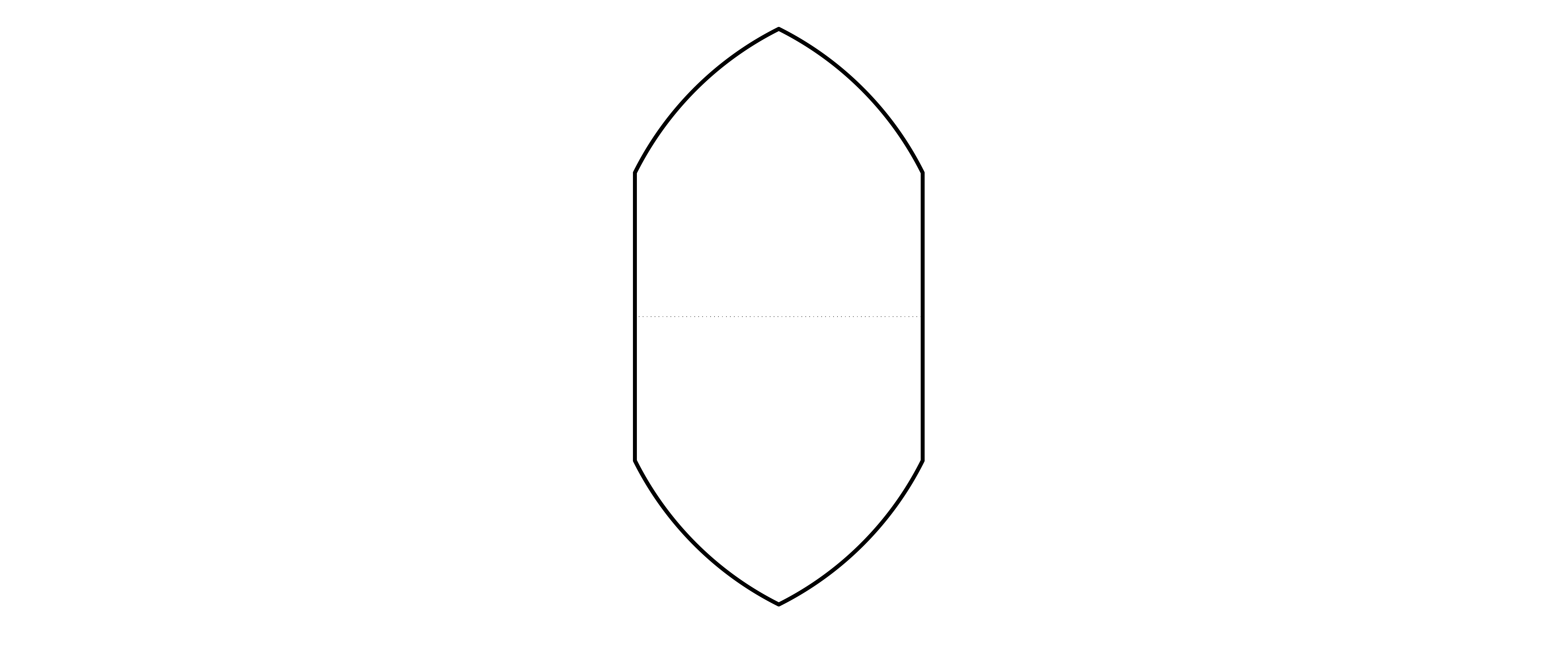}\\
  \caption{The set $Q$ and the corresponding minimizing bisection}\label{fig:optimomps}
\end{figure}

\begin{proposition}\label{prop:MioriPeriSegura}
Let $K\in\mathcal K^2$. Then,
\begin{equation} \label{eq:desigualdadmps}
\frac{\mathrm{A}(K)}{\widetilde{\D}_{B}(K)^2}\leq 2\arctan\bigg(\frac{3}{4}\bigg),
\end{equation}
with equality if $K=Q$.
\end{proposition}

Observe that inequality \eqref{eq:desigualdadmps} is an isodiametric-type inequality, in the sense of the classical
\emph{isodiametric inequality} of Bieberbach \cite{Bi}: given $K\in\K^2$, we have that
\begin{equation}\label{eq:bieberbach}
\A(K)\leq\frac{\pi}{4}\,\D(K)^2,
\end{equation}
with equality if and only if $K$ is an Euclidean disk.

Our Theorem~\ref{th:isodiametricIneq} below is an extension of Proposition \ref{prop:MioriPeriSegura}. % in two different ways:
On the one hand, we consider \emph{arbitrary bisections}, determined by curves which are \emph{not necessarily line segments}.
And on the other hand, we allow the regions of the bisections to have \emph{different areas}.
In other words, we focus on $\mathcal{B}(K)$ instead of $\mathcal{\widetilde{B}(K)}$.
This makes our approach completely general in this setting.
In Section~\ref{sec:IsodiamIneq} we shall prove the following result.

\begin{theorem}\label{th:isodiametricIneq}
Let $K\in\mathcal K^2$. Then,
\begin{equation}\label{eq:miniBisect}
\frac{\mathrm{A}(K)}{\D_B(K)^2}\leq 2\,\arctan\bigg(\frac{3}{4}\bigg),
\end{equation}
with equality if and only if $K=Q$.
\end{theorem}

\begin{remark}
\label{re:r}
Previous Theorem~\ref{th:isodiametricIneq} implies that, if we prescribe the enclosed area,
the convex body with the minimum possible value for $\D_B$ is precisely $Q$,
up to dilations and rigid motions (see Remark~\ref{re:inv}).
Furthermore, in Remark~\ref{re:futuro} we will characterize the minimizing bisections of $Q$:
a given bisection $(Q_1,Q_2)\in\mathcal B(Q)$ is minimizing if and only if
$Q_1\cap Q_2=l([-1,1])$ where
\[
\{l(-1),l(1)\}=\left\{\left(\pm\frac{1}{\sqrt{5}},0\right)\right\} \quad\text{and}\quad
l([-1,1])\subset\left\{(x_1,x_2)\in\mathbb R^2:x_1^2+\left(x_2\pm\frac{2}{\sqrt{5}}\right)^2\leq 1\right\}.
\]
\end{remark}

Surprisingly enough, the optimal set in the general situation,
described in Theorem \ref{th:isodiametricIneq},
is still the same set as in Proposition \ref{prop:MioriPeriSegura}.
This fact strengthens the idea that central symmetry
is an inherent property for this optimization problem.
On the other hand, we want to point out that the proof of our Theorem~\ref{th:isodiametricIneq}
cannot be carried out with the same arguments from \cite[Th.~6]{MPS},
where the authors focus on bisections given by \emph{line segments} and providing \emph{equal-area subsets}.
While the former restriction is not so significant (see our Lemma~\ref{lem:StraightLine}),
the later one entails a substantial reduction in the proof of \cite[Th.~6]{MPS},
since it directly implies that the optimal set can be supposed to be centrally symmetric.
In contrast, in the general case, the proof of our Theorem \ref{th:isodiametricIneq}
moves around the choice of two non-parallel supporting lines at the endpoints of the line segment
providing the minimizing bisection, and one cannot reduce to the simpler centrally symmetric case until the very last step.

Finally, it is worth mentioning that the questions regarding the maximum bisecting diameter
(firstly treated in \cite{MPS}) have given rise to several works in the last years.
In \cite{mejora} we can find some improvements for the centrally symmetric case,
and some related problems for divisions into three or more regions have been studied in \cite{extending, multi}.
Moreover, we also point out that these questions have been partially treated in surfaces of $\R^3$ \cite{surfaces1, surfaces2}.
Essentially, whenever there exists an isodiametric inequality, one can establish the corresponding
isodiametric inequality for bisections. Although we focus in this work in the planar case,
in Section \ref{sec:OtherSpaces} there are some considerations
on the isodiametric-type problem for bisections in $\mathbb R^n$,
and also in the spherical and the hyperbolic spaces.
%We encourage the interested reader to see in \cite{MPS} several applications of optimization over bisections in a convex domain.

Apart from studying the diameter, we also consider in this work the analogous problem for the \emph{width functional}
(which is, in some sense, the geometric functional \emph{reverse} to the diameter).
Recall that by replacing the diameter with the width in the classical isodiametric inequality, P\'al showed that
\begin{equation}\label{eq:palIneq}
\A(K)\geq\frac{1}{\sqrt{3}}\,\w(K)^2,
\end{equation}
with equality if and only if $K$ is an equilateral triangle \cite{Pal}.
Our aim is obtaining a similar \emph{isominwidth inequality} for bisections of a planar convex body.
For this purpose, given $K\in\K^2$, we can define, analogously to $\D_B(K)$, the \emph{infimum of the maximum bisecting width}
by
\begin{equation}
\label{eq:wB}
\w_B(K):=\inf_{(K_1, K_2)\in\mathcal{B}(K)}\max\{\w(K_1),\w(K_2)\}.
\end{equation}
We will prove in Section~\ref{sec:IsominwidthIneq} the following inequality.

\begin{theorem}\label{th:isominwidthIneq}
Let $K\in\mathcal K^2$. Then,
\begin{equation}\label{eq:widthBisection}
\frac{\mathrm{A}(K)}{\w_B(K)^2}\geq \frac{4}{\sqrt{3}},
\end{equation}
with equality if and only if $K$ is an equilateral triangle $\mathcal{T}$.
\end{theorem}

The techniques employed to prove Theorem \ref{th:isominwidthIneq} are based on a nice combination of
P\'al's inequality \eqref{eq:palIneq} and Bang's inequality on Tarski's plank problem \cite{Ba}.
We note that we will establish an isominwidth-type inequality for bisections in $\mathbb R^n$ in Section~\ref{sec:OtherSpaces}.

\begin{remark}
In analogy with Remark~\ref{re:r}, Theorem~\ref{th:isominwidthIneq} implies that if we prescribe the enclosed area,
the corresponding equilateral triangle $\mathcal{T}$ is the convex body with the largest possible value for $\w_B$.
That value is attained by the bisection determined by a line segment passing through the midpoints of two edges of $\mathcal{T}$
(see the proof of Theorem~\ref{th:isominwidthIneq}).
\end{remark}

\begin{remark}
\label{re:inv}
Notice that the quotients $\A(K)/\wt{\D}_B(K)^2$, $\A(K)/\D_B(K)^2$ and $\A(K)/\w_B(K)^2$
are invariant under dilations and rigid motions, due to the corresponding homogeneity of
the area, the diameter and the width functionals and the invariance under rigid motions.
Therefore, the uniqueness regarding the different optimal sets
has to be understood \emph{up to dilations and rigid motions}.
\end{remark}

On the other hand, the study of the \emph{reverse counterparts} to some geometric inequalities
has increasingly gained interest in the last years (see~\cite{Beh,ball,cha} and references therein).
In the case of the classical isodiametric inequality~\eqref{eq:bieberbach},
a reverse inequality cannot be stated directly since the \emph{isodiametric quotient} $\A(K)/\D(K)^2$,
for $K\in\K^2$, cannot be bounded from below by any constant different from $0$
(it suffices to consider very thin rectangles with area approaching zero).
However, Behrend treated this problem finding such lower bound for the family
of sets in $\K^2$ that maximizes that quotient in their affine class.
More precisely, we will say that $K\in\K^2$ is in
\emph{Behrend position} if
\[
\frac{\A(K)}{\D(K)^2}=\sup_{\phi\in End(\mathbb{R}^{2})}\frac{\A(\phi(K))}{\D(\phi(K))^2},
\]
where $End(\mathbb{R}^2)$ denotes the set of affine endomorphisms of $\mathbb{R}^2$  \cite{Beh}.
Therefore, if $K$ is in Behrend position, the above quotient achieves the maximum value among all the affine transformations of $K$.
This approach allows to obtain an interesting reverse isodiametric inequality: for every $K\in\K^2$ in Behrend position, we have that
\begin{equation}\label{eq:BehrendIneq}
\A(K)\geq\frac{\sqrt{3}}{4}\,\D(K)^2,
\end{equation}
with equality if and only if $K$ is an equilateral triangle \cite{Beh}. Moreover, if we restrict
$K$ to be centrally symmetric (that is, $K=x-K$ for some $x\in\mathbb R^2$), then
\begin{equation}\label{eq:Behrend0symm}
\A(K)\geq\frac12\,\D(K)^2,
\end{equation}
with equality if and only if $K$ is a square (\cite{Beh}, see also~\cite{GMS}).
Following these ideas (also used by Ball for obtaining the first reverse isoperimetric inequality \cite{ball}),
we will establish an analogous inequality to \eqref{eq:BehrendIneq} for the infimum of the maximum bisecting diameter.
In order to do this, we will say that $K\in\K^2$ is in \emph{Behrend-bisecting position} if
\begin{equation}
\label{def:BBP}
\frac{\A(K)}{\D_B(K)^2}=\sup_{\phi\in End(\mathbb{R}^2)}\frac{\A(\phi(K))}{\D_B(\phi(K))^2}.
\end{equation}
In Section~\ref{sec:RevIsodiamIneq} we give some necessary conditions
for a set $K$ to be in Behrend-bisecting position.
In particular, and contrary to intuition,
we will see that an equilateral triangle \emph{is not in Behrend-bisecting position}.
In fact, Proposition~\ref{prop:TriangleOptimal} gives a characterization of the unique triangle in Behrend-bisecting position,
being an isosceles triangle whose different angle equals $\arccos(\sqrt{2/3})$.
Apart from this, our Theorem~\ref{th:RevIsodiamBisectIneq} establishes the following reverse isodiametric inequality,
which is not sharp in general.

\begin{theorem}\label{th:RevIsodiamBisectIneq}
Let $K\in\K^2$ be in Behrend-bisecting position. Then,
\begin{equation}\label{eq:RevIsodBisecting}
\frac{\A(K)}{\D_B(K)^2}\geq\frac{\sqrt{3}}{4}.
\end{equation}
\end{theorem}

Moreover, the restriction to centrally symmetric convex bodies in Behrend-bisecting position
allows to improve inequality~\eqref{eq:RevIsodBisecting},
as shown in our Theorem~\ref{th:RevIsodiamBisect0symmIneq}.

\begin{theorem}\label{th:RevIsodiamBisect0symmIneq}
Let $K\in\K^2$ be centrally symmetric and in Behrend-bisecting position. Then,
\begin{equation}\label{eq:RevIsodiamBisect0symm}
\frac{\A(K)}{\D_B(K)^2}\geq \frac{\sqrt{3}}{2}.
\end{equation}
\end{theorem}

In this setting, we also remark that Proposition~\ref{prop:ParallelogramOptimal}
characterizes the parallelograms in Behrend-bisecting position:
these sets are precisely the rectangles formed by joining two squares by a common edge.  %with an edge twice as large as the other edge.
Note that, in particular, a parallelogram formed by joining two equilateral triangles is not in Behrend-bisecting position.

%\textcolor{red}{As before, and contrary to intuition, a parallelogram formed by the union of two equilateral
%triangles whose intersection is a common edge, \emph{is not in Behrend-bisecting position}.
%Instead, the rectangle with an edge twice as large as the other edge is the parallelogram
%in Behrend-bisecting position (see Proposition \ref{prop:ParallelogramOptimal}).}

We would also like to note that the proof of Theorem \ref{th:RevIsodiamBisectIneq}
(resp., Theorem~\ref{th:RevIsodiamBisect0symmIneq}) is inspired in the proof of \cite[Th.~1.4]{GMS} (resp.,~\cite[Prop.~1.3]{GMS})
to reprove Behrend's inequality \eqref{eq:BehrendIneq}.
In essence, we provide the corresponding necessary condition of Behrend-bisecting position, see Lemma \ref{lem:behrendBisectPos}
(resp., the necessary condition of ~Behrend-bisecting position for centrally symmetric convex bodies, see Lemma \ref{lem:Behrend-bisect-0symm}),
which differs from the conditions for being in Behrend position (see Proposition~\ref{prop:behrendPos}).

The same spirit of the previous results leads us to study
a \emph{reverse isominwidth inequality} for minimizing bisections, of type
$\A(K)/\w_B(K)^2\leq\alpha$, for some $\alpha\in\mathbb{R}$.
We will follow an approach similar to \cite{GMS}, considering again affine classes of sets in $\K^2$.
%As for the diameter functional, appropriate thin rectangles show that, in order to obtain a
%non-trivial inequality, we need to focus on the affine classes,
%following again an approach similar to \cite{GMS}. %(which is partially based on the one used in \cite{Beh}).
In this sense, recall that $K\in\K^2$ is in \emph{isominwidth optimal position} if
\begin{equation}
\label{eq:iop}
\frac{\A(K)}{\w(K)^2}=\inf_{\phi\in End(\mathbb{R}^{2})}\frac{\A(\phi(K))}{\w(\phi(K))^2}.
\end{equation}
The restriction to these suitable affine representatives of planar convex bodies yields,
as in the case of the diameter functional, to the following result:
for any set $K\in\K^2$ in isominwidth optimal position, it holds that
\begin{equation}\label{eq:GMSchymuraIneq}
\A(K)\leq\w(K)^2,
\end{equation}
with equality if and only if $K$ is a square \cite[Th.~1.6]{GMS}.
Our aim is obtaining an analogous inequality to \eqref{eq:GMSchymuraIneq}
for the infimum of the maximum bisecting width
for sets \emph{in a certain special position}.
Thus, given $K\in\K^2$, we will say that $K$ is in \emph{isominwidth-bisecting position} if
\begin{equation}
\label{eq:isominwidthoptimalposition}
\frac{\A(K)}{\w_B(K)^2}=\inf_{\phi\in End(\mathbb R^{2})}\frac{\A(\phi(K))}{\w_B(\phi(K))^2}.
\end{equation}
We will derive in Section~\ref{sec:RevIsominwidthIneq}
some necessary and sufficient conditions for being in isominwidth-bisecting position,
concluding with our Theorem~\ref{th:RevIsominwidthBisectIneq},
which follows again from Bang's inequality~\cite{Ba}
and inequality~\eqref{eq:GMSchymuraIneq}.

\begin{theorem}\label{th:RevIsominwidthBisectIneq}
Let $K\in\K^2$ be in isominwidth-bisecting position. Then,
\begin{equation}\label{eq:RevIsominwidthBisecting}
\frac{\A(K)}{\w_B(K)^2}\leq 4,
\end{equation}
with equality if and only if $K$ is a square $\mathcal{C}$.
\end{theorem}

\begin{remark}
We point out that $\w_B(\mathcal{C})$ is attained by the bisection
determined by a segment parallel to an edge of $\mathcal C$ dividing $\mathcal{C}$ into two equal-area subsets.
\end{remark}

%\textcolor{red}{Notice that the proof ideas of Theorem \ref{th:RevIsominwidthBisectIneq}
%are based on \eqref{eq:GMSchymuraIneq}. Basically, it depends on the relation between
%Bang's inequality on Tarski's plank problem \cite{Ba} and \eqref{eq:GMSchymuraIneq}.}

The paper is organized as follows.
In Section \ref{sec:PropMinBisect} we obtain some general properties of the minimizing bisections
for the maximum bisecting diameter and the maximum bisecting width.
In particular, Lemma~\ref{lem:StraightLine} shows that there always exists a minimizing bisection given by a line segment,
which allows to focus only on this type of bisections along this work.
In Section \ref{sec:IsodiamIneq} we prove Theorem \ref{th:isodiametricIneq},
determining the corresponding optimal set (of fixed area) for the maximum bisecting diameter
by a constructive argument.
Section \ref{sec:IsominwidthIneq} is devoted to show Theorem \ref{th:isominwidthIneq},
which follows directly from Lemma~\ref{lem:WidthBisect=Half}.
Sections~\ref{sec:RevIsodiamIneq} and~\ref{sec:RevIsominwidthIneq} treat the reverse inequalities
under the approach of affine representatives of planar convex bodies.
In Section \ref{sec:RevIsodiamIneq} we demonstrate Theorem~\ref{th:RevIsodiamBisectIneq},
which requires a detailed study concerning the Behrend-bisecting position,
%In Lemma~\ref{lem:isodiamMaximum} we prove the existence of sets in Behrend-bisecting position for the diameter functional,proving Theorem~\ref{th:RevIsodiamBisectIneq}.
and Section \ref{sec:RevIsominwidthIneq} contains the proof of Theorem~\ref{th:RevIsominwidthBisectIneq}.
% Finally, Section~\ref{sec:OtherSpaces} shows some comments on the extensions of the considered problems to higher dimension, as well as to the spherical and the hyperbolic spaces.
Finally, in Section~\ref{sec:OtherSpaces} we explore how to extend
the isodiametric and isominwidth inequalities for bisections
in the Euclidean space of higher dimension (Subsections~\ref{subsec:rn} and~\ref{subsec:rn2}),
as well as in the spherical and hyperbolic spaces (Subsection~\ref{subsec:hn}).

\subsection*{Notation} We now establish some notation used throughout this paper.
The Euclidean distance in $\mathbb R^2$ will be denoted by $d$,
and the Hausdorff distance for planar compact sets will be denoted by $d_\mathcal{H}$.
%The \emph{vectors of the canonical basis} of $\R^2$ will be $e_1=(1,0)$ and $e_2=(0,1)$.
Given two points $x,y\in\R^2$, $[x,y]$ will represent the \emph{line segment} with endpoints $x$ and $y$.
For every $K\in\K^2$, $\mathrm{Ext}(K)$ will stand for the set of \emph{extreme points} of $K$,
i.e., if $x\in\mathrm{Ext}(K)$, then $x\in [y,z]\subset K$ implies $x=y$ or $x=z$.
For any planar compact set $A$, we denote
by $\conv(A)$ and $\lin(A)$ the
\emph{convex hull} and the \emph{linear hull} of $A$, respectively.
Moreover, %if $A$ is a planar set,
we denote by $A^\bot$ the \emph{orthogonal complement} of $A$, i.e.,
$A^\bot=\{x\in\mathbb R^2:x^Ty=0,\,\forall y\in A\}$.

Furthermore, for $K\in\K^2$ and $u\in\R^2\setminus\{0\}$, the \emph{Steiner symmetrization} $s_u(K)$ of $K$
with respect to $\lin(u)$ is defined as the only symmetric set with respect to
$\lin(u)$ such that each segment $(tu+u^\bot)\cap s_u(K)$ has the same length than $(tu+u^\bot)\cap K$, for every $t\in\R$  \cite{BF, SY}.
It is well known that %if $K\in\K^2$ and $u\in\R^2\setminus\{0\}$ then
$s_u(K)\in\K^2$ and
\begin{equation}\label{eq:SteinerSymm}
\A(s_u(K))=\A(K),\quad\D(s_u(K))\leq\D(K).
\end{equation}

\section{Properties of minimizing bisections}\label{sec:PropMinBisect}

In this section we will obtain some interesting properties
for the minimizing bisections of the two functionals $\D_B(K)$, $\w_B(K)$ we are considering.
Lemma~\ref{lem:StraightLine} shows that there is always one of these bisections given by a line segment, extending \cite[Prop.~4]{MPS},
and Lemma~\ref{lem:EqualValueBisection} further proves that the subsets of that bisection are \emph{in equilibrium}, in some sense.
%minimizing bisections always provide, in some sense, two regions which are \emph{in equilibrium}.
Besides, we also show in Lemma~\ref{lem:minimum}
that the infimums in \eqref{eq:db} and \eqref{eq:wB} are attained (and so they are actually minimums).

\begin{lemma}\label{lem:StraightLine}
Let $K\in\K^2$ and $\rho>0$.
For any bisection of $K$ with maximum bisecting diameter (or width) equal to $\rho$,
there exists another bisection of $K$ given by a line segment with maximum bisecting diameter (or width) smaller than or equal to $\rho$.
\end{lemma}

\begin{proof}
Consider $(K_1,K_2)\in\mathcal{B}(K)$ determined by an injective continuous curve $l:[-1,1]\rightarrow K$ with $l(-1)$, $l(1)\in\bd(K)$.
Suppose that $\max\{\D(K_1),\D(K_2)\}=\rho$ (or $\max\{\w(K_1),\w(K_2)\}=\rho$).
Call $M_1:=\bd(K)\cap K_1$ and $M_2:=\bd(K)\cap K_2$.
Since $M_i\subset K_i$, then $\D(M_i)\leq\D(K_i)$ and $\w(M_i)\leq\w(K_i)$, for $i=1,2$.

Notice that the line segment $[l(-1),l(1)]$ determines $\conv(M_i)$, $i=1,2$. %joining the points $\gamma(1)$ and $\gamma(-1)$.
We claim that $\D(M_i)=\D(\conv(M_i))$, $i=1,2$.
On the one hand, $M_i\subset\conv(M_i)$ implies that $\D(M_i)\leq\D(\conv(M_i))$.
And on the other hand, $\mathrm{Ext}(\conv(M_i))\subset M_i$,
since any point in $\mathrm{Ext}(\conv(M_i))$ cannot be expressed as an strict convex combination of points in $\conv(M_i)$. %by definition of extreme point.
Furthermore, since the diameter in $\mathbb R^2$ is always attained by a pair of extreme points,
it follows that
$$\D(\conv(M_i))=\D(\mathrm{Ext}(\conv(M_i)))\leq \D(M_i).$$
We also have that $\w(M_i)=\w(\conv(M_i))$, $i=1,2$,
as a direct consequence of the fact that $M_i$ is contained between two parallel lines if and only if
$\conv(M_i)$ is contained between those lines.
Then, $\conv(M_1)$, $\conv(M_2)$ are two subsets of $K$ providing a bisection of $K$,
%($K=\conv(M_1)\cup\conv(M_2)$ and $\conv(M_1)\cap\conv(M_2)=[\gamma(-1),\gamma(1)]$),
satisfying
\[
\max\{\D(\conv(M_1)),\D(\conv(M_2))\}\leq\max\{\D(K_1),\D(K_2)\}=\rho,
\]
as well as
\[
\max\{\w(\conv(M_1)),\w(\conv(M_2))\}\leq\max\{\w(K_1),\w(K_2)\}=\rho,
\]
and so, $(\conv(M_1)$, $\conv(M_2))$ is a bisection of $K$
given by a line segment with maximum bisecting diameter (or width)
smaller than or equal to $\rho$, as stated.
\end{proof}

\begin{lemma}
\label{lem:minimum}
Let $K\in\K^2$. Then,
\[
\D_B(K)=\min_{(K_1,K_2)\in\mathcal{B}(K)}\max\{\D(K_1),\D(K_2)\},
\]
and
\[
\w_B(K)=\min_{(K_1,K_2)\in\mathcal{B}(K)}\max\{\w(K_1),\w(K_2)\}.
\]
\end{lemma}

\begin{proof}
We will focus on $\D_B(K)$, since the case of $\w_B(K)$ is analogous.
Note that Lemma~\ref{lem:StraightLine} allows to consider only bisections by line segments in order to compute $\D_B(K)$.
%Note that Lemma~\ref{lem:StraightLine} shows that $\D_B(K)$ and $\w_B(K)$ are infimums taken over bisections by line segments.
Then, in view of~\eqref{eq:db},
%\[
%\rho=\inf_{\{K_1,K_2\}\in\mathcal{B}(K)}\max\{\D(K_1),\D(K_2)\}.
%\]
let $\{[a_i,b_i]\}_{i\in\mathbb N}\subset  K$ be a sequence of line segments providing bisections $(K_{1,i},K_{2,i})$ of $K$,
%each of them with subsets $\{K_{1,i},\,K_{2,i}\}$,
such that
\[
\D_B(K)=\lim_{i\rightarrow\infty}\max\{\D(K_{1,i}),\D(K_{2,i})\}.
\]

Since $\{K_{1,i}\}_{i\in\mathbb{N}} \subset K$ is an absolutely bounded sequence of convex bodies,
\emph{Blaschke Selection Theorem} \cite[Th.~1.8.7]{Sch} implies the existence of a %convergent
subsequence $\{K_{1,i_j}\}$ and a subset $K_1\in\K^2$ such that $K_1\subset K$ and
$\displaystyle{\lim_{i_j\rightarrow\infty}K_{1,i_j}=K_1}$ in Hausdorff metric.
Considering now the corresponding subsequence $\{K_{2,i_j}\}$ of $\{K_{2,i}\}$,
it is clear that  $\displaystyle{\lim_{i_j\rightarrow\infty}K_{2,i_j}=K_2}$,
where $K_2=\overline{K\setminus K_1}$.
Note that, without loss of generality, we can assume that the subsequences are the sequences themselves.
In particular, we also obtain that $\displaystyle{\lim_{i\rightarrow\infty}[a_i,b_i]=[a,b]}$,
for certain $a,b\in K$, with $[a,b]=K_1\cap K_2$, and so $(K_1,K_2)$ is a bisection of $K$.
Since the diameter %(as well as the minimum width)
is a continuous functional with respect to Hausdorff metric, we have that
$\displaystyle{\lim_{i\rightarrow\infty}\D(K_{1,i})=\D(K_1)}$ and $\displaystyle{\lim_{i\rightarrow\infty}\D(K_{2,i})=\D(K_2)}$,
which implies that $\D_B(K)=\max\{\D(K_1),\D(K_2)\}$, as stated.
%i.e., $\{K_1,K_2\}$ is a minimizing bisection for the maximum bisecting diameter.
\end{proof}

% referencia donde venga que el diametro es continuo respecto la dist. de Hausdorff: Bonessen-Fenchel

\begin{lemma}\label{lem:EqualValueBisection}
Let $K\in\K^2$.
%\textcolor{red}{If $(K_1,K_2)$ is a bisection of $K$ minimizing the maximum bisecting diameter (or width)
%with $K_1\cap K_2$ being a line segment parallel to a line $L$, then there exists another bisection $(K_1',K_2')$
%of $K$ minimizing the maximum bisecting diameter (or width) with $K_1'\cap K_2'$ being another line segment parallel to $L$ too,
%with $\D_B(K)=\D(K_1')=\D(K_2')$ (or $\w_B(K)=\w(K_1')=\w(K_2')$)}.
There exists a bisection $(K_1,K_2)$ of $K$ minimizing the maximum bisecting diameter (or width) of $K$
such that $\D_B(K)=\D(K_1)=\D(K_2)$ (or $\w_B(K)=\w(K_1)=\w(K_2)$).
\end{lemma}

\begin{proof}
This is a consequence of the continuity of the diameter and the width functionals.
Taking into account Lemmas~\ref{lem:StraightLine} and~\ref{lem:minimum},
let $(K_1,K_2)$ be a bisection of $K$ minimizing the maximum bisecting diameter (or width),
determined by the line segment $L=K_1\cap K_2$.
Fix $u$ an orthogonal vector to $\mathrm{span}(L)$, and let $t_1<0<t_2$ be such that $K\cap(tu+L)\neq\emptyset$
when and only when $t\in[t_1,t_2]$.
Moreover let $K_1^t=K\cap\{su+L:s\in[t_1,t]\}$ and $K_2^t=K\cap\{su+L:s\in[t,t_2]\}$, for every $t\in[t_1,t_2]$, so that $K_i^0=K_i$, $i=1,2$.
In particular, we have that $K\cap(t_iu+L)\subset\mathrm{bd}(K)$, $i=1,2$, and thus $K_1^{t_2}=K_2^{t_1}=K$.

For $i=1,2$, let $f_i,g_i:[t_1,t_2]\rightarrow[0,\D(K)]$
be such that $f_1(t)=\D(K_1^t)$, $f_2(t)=\D(K_2^t)$,
and $g_1(t)=\w(K^t_1)$, $g_2(t)=\w(K_2^t)$.
By direct inclusion of sets, we have that $f_1$ and $g_1$ are non-decreasing,
whereas $f_2$ and $g_2$ are non-increasing. Moreover, these four functions are continuous,
with $f_1(t_2)=\D(K)=f_2(t_1)$ and $g_1(t_2)=\w(K)=g_2(t_1)$.

If $f_1(0)=f_2(0)$ (resp.,~$g_1(0)=g_2(0)$), then $(K_1,K_2)$ is a minimizing bisection with $\D_B(K)=\D(K_1)=\D(K_2)$
(resp.,~$\w_B(K)=\w(K_1)=\w(K_2)$), as desired.
Otherwise, let us suppose without loss of generality that $f_1(0)<f_2(0)=\D(K_2)=\D_B(K)$
(resp.,~$g_1(0)<g_2(0)=\w(K_2)=\w_B(K)$). Since
\[
f_1(t_2)=\D(K)\geq\D(K_2^{t_2})=f_2(t_2),
\]
(resp.,~$g_1(t_2)\geq g_2(t_2)$),
\emph{Bolzano Theorem} implies that there exists $t_0\in[0,t_2]$ such that
$f_1(t_0)=f_2(t_0)$ (resp.,~$g_1(t_0)=g_2(t_0)$).
By using the monotonicity of the functions, we have that
\[
\D(K_1)=f_1(0)\leq f_1(t_0)=f_2(t_0)\leq f_2(0)=\D(K_2)=\D_B(K)
\]
and
\[
\w(K_1)=g_1(0)\leq g_1(t_0)=g_2(t_0)\leq g_2(0)=\w(K_2)=\w_B(K),
\]
thus $\D(K_1^{t_0})=\D(K_2^{t_0})\leq\D_B(K)$
(resp.,~$\w(K_1^{t_0})=\w(K_2^{t_0})\leq\w_B(K)$),
and hence $(K_1^{t_0}, K_2^{t_0})$
is a minimizing bisection of $K$ providing subsets of equal diameters (or widths), as desired.
\end{proof}

\begin{remark}
In fact, previous Lemma~\ref{lem:EqualValueBisection} proves that for every minimizing bisection determined by a line segment $l$,
there exists another minimizing bisection $(K_1',K_2')$ with $\D(K_1')=\D(K_2')$ (or $\w(K_1')=\w(K_2')$) and determined by a line segment parallel to $l$.
\end{remark}

\begin{remark}
A minimizing bisection $(K_1,K_2)$ for $D_B$ with subsets of equal \emph{diameters} as in Lemma \ref{lem:EqualValueBisection}
might be degenerate, that is, $K_1$ or $K_2$ might be reduced to a line segment.
For instance, let $\mathcal T\in\mathcal K^2$ be an equilateral triangle of vertices $p_i$, $i=1,2,3$. Then
$(\mathcal T,[p_1,p_2])$ is a minimizing bisection with $\D_B(\mathcal T)=\D(\mathcal T)=\D([p_1,p_2])$.
This is not the case for the minimizing bisections for $\w_B$ with subsets of equal widths,
which have to split any convex body into
two non-degenerate subsets, since the width of a line segment trivially vanishes.
\end{remark}

\section{The isodiametric inequality}\label{sec:IsodiamIneq}
In this section we will prove our Theorem~\ref{th:isodiametricIneq}, providing an isodiametric-type inequality involving $D_B$.
As we will see, the proof of this result is constructive, yielding the corresponding
optimal set. We first prove Lemma~\ref{lem:Symmetry}.

\begin{lemma}\label{lem:Symmetry}
There exists a maximizer $K_0\in\K^2$ of the quotient $\A(K)/\D_B(K)^2$,
with $\D_B(K_0)$ attained by a bisection of $K_0$ determined by a line segment $[(-a,0),(a,0)]$, $a>0$,
such that $K_0$ is symmetric with respect
to the line $L=\{(x_1,x_2)\in\R^2:x_1=0\}$.
\end{lemma}

\begin{proof}
Consider the supremum
$$\gamma:=\sup_{K\in\K^2}\frac{\A(K)}{\D_B(K)^2}>0.$$
Taking into account that the area and the diameter functionals
are homogeneous with respect to dilations (see Remark~\ref{re:inv}),
we can normalize to unit area and so
$$\gamma=\sup_{\substack{K\in\K^2 \\[0.6mm] \A(K)=1}}\frac{1}{\D_B(K)^2}.$$

By the definition of supremum, let $\{K_n\}$ be a sequence in $\K^2$
with $\lim_{n\rightarrow\infty}\D_B(K_n)^2=\gamma^{-1}$, and $\A(K_n)=1$ for every $n\in\mathbb N$.
We claim that $\D_B(K_n)\leq C$, for certain $C>0$.
Otherwise, $\D_B(K_n)$ would tend to infinity, which contradicts the positivity of $\gamma$.
Hence $\D(K_n)\leq 2\,\D_B(K_n)\leq 2\,C$, which implies that $\{K_n\}$ is a bounded sequence.
After a suitable translation of each $K_n$, we can assume that $\{K_n\}$ is absolutely bounded.
Hence, by Blaschke Selection Theorem, there exists a subsequence of $\{K_n\}$
%(which will be denoted as the original one)
convergent to some $\widetilde{K}\in\K^2$ in Hausdorff metric.
By continuity, it follows that
$$\frac{1}{\D_B(\widetilde{K})^2}=\gamma,$$ % el area de $\widetilde{K}$ es uno, por ser el límite de la sucesion
and so $\widetilde{K}$ is a maximizer of the quotient.

By Lemma~\ref{lem:StraightLine}, we can now suppose without loss of generality that
$\D_B(\widetilde{K})$ is given by a bisection $(K_1,K_2)$ of $\widetilde{K}$,
with $K_1=\wt{K}\cap H^+$, $K_2=\wt{K}\cap H^-$, $K_1\cap K_2=[(-a,0),(a,0)]$, for some $a\in[0,\D(K)/2]$,
where $H^+=\{(x_1,x_2)\in\R^2:x_2\geq 0\}$ and $H^-=\{(x_1,x_2)\in\R^2:x_2\leq 0\}$.
Call $e_2=(0,1)\in\mathbb R^2$. By applying Steiner symmetrization $s_{e_2}$ with respect to the vertical line
$\mathrm{span}(e_2)$, we easily get that
$s_{e_2}(\wt{K})=s_{e_2}(K_1)\cup s_{e_2}(K_2)$ and $s_{e_2}(K_1)\cap s_{e_2}(K_2)=[(-a,0),(a,0)]$.
Denoting by $K_0:=s_{e_2}(\wt{K})$ and $K_{0,i}:=s_{e_2}(K_i)$,
we have by~\eqref{eq:SteinerSymm} that $\A(K_{0,i})=\A(K_i)$ and $\D(K_{0,i})\leq\D(K_i)$, $i=1,2$,
and so $\A(K_0)=\A(\wt{K})$ and $\D_B(K_0)\leq\D_B(\wt{K})$.
Since $\wt{K}$ is a maximizer of the quotient $\A(K)/\D_B(K)^2$,
then necessarily $K_0$ is also a maximizer, which possesses the desired symmetry by construction.
\end{proof}

\begin{proof}[Proof of Theorem \ref{th:isodiametricIneq}]
Consider an arbitrary $K\in\K^2$. We will apply several transformations to $K$,
without decreasing the enclosed area, arriving at the end of the process at the set $Q\in\K^2$,
which satisfies $\A(Q)\geq\A(K)$ and $\D_B(Q)=\D_B(K)$. This will prove the maximality of $Q$.
%Let $\wt{K}\in\K^2$ be a maximizer of the isodiametric quotient $\A(K)/\D_B(K)^2$. We will prove that there exists another convex set $K_0$ whose quotient is strictly greater than
%$\displaystyle{A(K)}{D_B(K)^2}$, $\A(\wt{K})/\D_B(\wt{K})^2$ whenever $\wt{K}$ is different from $K_0$
%(up to dilations and rigid motions, see Remark~\ref{re:inv}), which implies that the maximizer must be precisely $K_0$.

Let us suppose without loss of generality that $(K_1,K_2)$ is a bisection of $K$ providing $\D_B(K)$,
with $K_1=K\cap H^+$, $K_2=K\cap H^-$, $K_1\cap K_2=[(-a,0),(a,0)]$, for some $a\in[0,\D_B(K)/2]$,
where $H^+=\{(x_1,x_2)\in\R^2:x_2\geq 0\}$ and $H^-=\{(x_1,x_2)\in\R^2:x_2\leq 0\}$, in view of Lemma~\ref{lem:Symmetry}.
We can also assume that $K$ is symmetric with respect to the vertical line $L=\{x\in\R^2:x_1=0\}$
and, by Lemma \ref{lem:EqualValueBisection},
that $\D(K_1)=\D(K_2)=\D_B(K)$.

Since $K$ is convex and compact, and $(a,0)\in\mathrm{bd}(K)$, then there exists a supporting
line $M_+$ to $K$ at $(a,0)$. Due to the symmetry of $K$, the symmetric
line of $M_+$ with respect to $L$ is also a supporting line at $(-a,0)$, namely $M_-$.
By flipping the situation if necessary, we can suppose that
the slope of $M_+$ is non-negative, and so $M_+=\{(x_1,x_2)\in\R^2:x_2=m(x_1-a)\}$, for some $m\geq 0$.
Additionally, let $B_{\pm}=B((\pm a,0),\D_B(K))$ be the closed balls centered at $(\pm a,0)$ and of radius $\D_B(K)$.
Since $\D(K_i)=\D_B(K)$ and $(\pm a,0)\in K_i$, it follows that $K_i$ is necessarily contained in the symmetric lens $B_+\cap B_-$, for $i=1,2$.

If $M_+$ is not vertical, we have that $K_2$ is contained in the triangle $T$ determined by $M_+$, $M_-$, and the horizontal line
$\{(x_1,x_2)\in\R^2:x_2=0\}$.
\begin{figure}[ht]
  \includegraphics[width=0.4\textwidth]{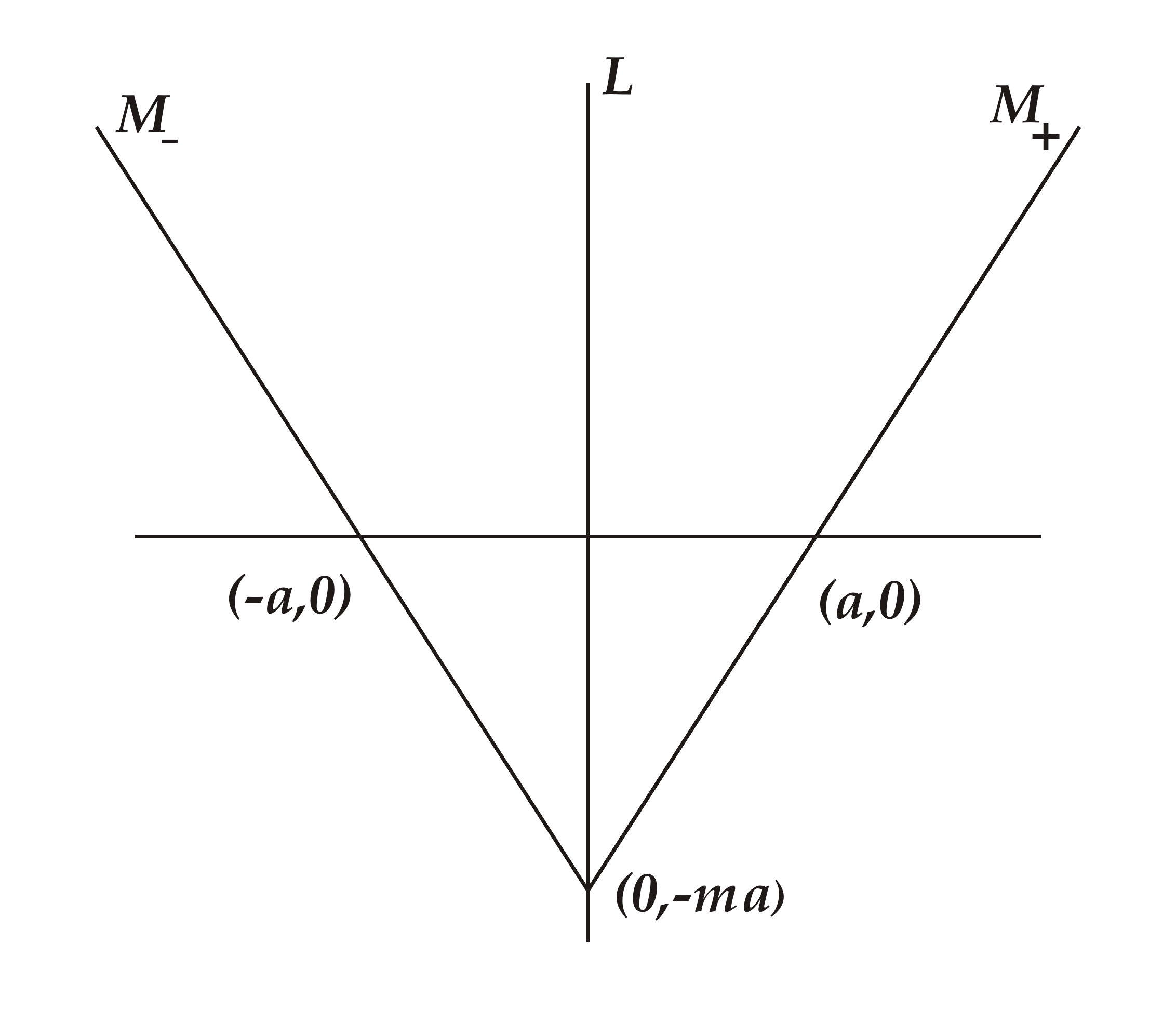}\\
  \caption{If $M_+$ is not vertical, then $K_2$ is contained in the triangle $T$} \label{fig:supportinglines}
\end{figure}
Then $\D(T)=\max\{2a,\delta\}\geq\D(K_2)=\D_B(K)$, where $\delta=d((a,0),(0,-ma))=a\sqrt{1+m^2}$.
We will distinguish two possibilities.
If $2a>\delta$, then $2a=\D(T)\geq\D_B(K)$ (and so $\D_B(K)=2a$). %, which implies that $D=2a$.
In this case, it is straightforward checking that the area of $B_+\cap B_-$ equals
$\displaystyle{\D_B(K)^2\,\frac{4\pi-3\sqrt{3}}{6}}$,
and so
\begin{equation}
\label{eq:2a=D}
\frac{\A(K)}{\D_B(K)^2}\leq\frac{\A(B_+\cap B_-)}{\D_B(K)^2}=\frac{4\pi-3\sqrt{3}}{6}.
\end{equation}
On the other hand, if $2a\leq \delta$, then $\delta=\D(T)\geq\D_B(K)$, which implies that $m\geq a^{-1}\sqrt{\D_B(K)^2-a^2}$.
Let us estimate the isodiametric quotient of $K$ in this case.
%(otherwise, the diameter of $T$ would be strictly smaller than $\D_B(\wt{K})$).

Let $R(a,m)$ be the planar region contained between $M_+$, $M_-$, $B_+$ and $B_-$, with the dependance
on $a$ and $m$ explained above.
Since $K\subseteq R(a,m)$, then $\A(K)\leq \A(R(a,m))$.
Moreover, let $R(a,+\infty)$ be the planar region contained between $B_+$, $B_-$
and the vertical lines passing through $(\pm a,0)$.
Let us check that
%We will compare the areas of $R(a,+\infty)$ and $R(a,m)$, proving that
$\A(R(a,m))<\A(R(a,+\infty))$, for every $m\geq a^{-1}\sqrt{\D_B(K)^2-a^2}$ (and $2a\leq\delta$).
Due to the symmetry of these regions, we can focus on the corresponding areas contained in $\{(x_1,x_2)\in\R^2:x_1\geq0\}$. %the right-hand side of $L$.
The only region $R_1$ (resp., $R_2$) contained in $R(a,m)$ (resp., $R(a,+\infty)$) which is not in $R(a,+\infty)$
(resp., $R(a,m)$) is the one contained
between $M_+$, $(a,0)+L$, $B_-$, and $\{(x_1,x_2)\in \R^2:x_2\geq 0\}$ (resp., $\{(x_1,x_2)\in \R^2:x_2\leq 0\}$).
It can be checked that the condition $m\geq a^{-1}\sqrt{\D_B(K)^2-a^2}$ implies that the rotation centered at $(a,0)$ of angle $\pi$
maps strictly $R_1$ onto $R_2$, and so $\A(R(a,m))<\A(R(a,+\infty))$. %for $m\geq\sqrt{\D^2-a^2}/a$.
Note also that the construction of $R(a,+\infty)$ implies that $\D_B(R(a,+\infty))=\D_B(K)$
(the bisection of $R(a,+\infty)$ given by the subsets $R_+=R(a,+\infty)\cap H^+$ and $R_-=R(a,+\infty)\cap H^-$
satisfies $\D(R_+)=\D(R_-)=\D_B(K)$, see \cite{MPS}).

Let us now compute the maximum value for $\displaystyle{\A(R(a,+\infty))/\D_B(K)^2}$, when $a>0$.
It is straightforward checking that
\begin{align*}
\A(a):&=\A(R(a,+\infty)) =4\int_0^a\sqrt{\D_B(K)^2-(x+a)^2}\,dx\\
&=2\Bigg(2a\sqrt{\D_B(K)^2-4a^2}-a\sqrt{\D_B(K)^2-a^2}+
\D_B(K)^2\arctan\Bigg(\frac{2a}{\sqrt{\D_B(K)^2-4a^2}}\Bigg)\\
&-\D_B(K)^2\arctan\Bigg(\frac{a}{\sqrt{\D_B(K)^2-a^2}}\Bigg)\Bigg).
\end{align*}
%Then
%\begin{equation}
%\label{eq:intermedia}
%\frac{\A(\wt{K})}{D_B(\wt{K})^2} \leq \frac{\A(R(a,+\infty))}{D_B(\wt{K})^2}.
%\end{equation}
%We now observe that the closest point in $\mathrm{bd}(B_-)$ to $(a,0)$ is $(D-a,0)$, and moving along
%the arc $\mathrm{bd}(B_-)$ towards $(-D-a,0)$ increases the distance. This directly means that a
%rotation of $\pi$ radians around $(a,0)$ lets $R_1$ be strictly contained in $R_2$, and thus
%$\A(R(a,m))<\A(R(a,+\infty))$ for $m<+\infty$.
For simplicity, call $\displaystyle{b=a/\D_B(K)}$
(which corresponds to a normalization for having $\D_B(K)$ equal to 1 by an appropriate dilation).
Then, well-known properties of dilations gives
\begin{align*}
\A(b)&=%\A\bigg(\frac{1}{D_B(\wt{K})}\,R(a,+\infty)\bigg)=\frac{\A(R(a,+\infty))}{D_B(\wt{K})^2}=
2\left(2b\sqrt{1-4b^2}-b\sqrt{1-b^2}+\arctan\left(\frac{2b}{\sqrt{1-4b^2}}\right)-\arctan\left(\frac{b}{\sqrt{1-b^2}}\right)\right),
\end{align*}
which attains its maximum value (as a function on $b$) only at $b=1/\sqrt{5}$, %is a maximum since $A'(1/\sqrt{5})$ and $A''(1/\sqrt{5})<0$.
and so, for any $b>0$,
$$\displaystyle{\A(b)\leq\A(1/\sqrt{5})=2\,\arctan\bigg(\frac{3}{4}\bigg)}.$$
Thus
\begin{equation}
\label{eq:2}
\frac{\A(K)}{\D_B(K)^2} \leq \frac{\A(R(a,+\infty))}{\D_B(K)^2}\leq2\,\arctan\bigg(\frac{3}{4}\bigg),
\end{equation}
which gives a bound greater than the one obtained in \eqref{eq:2a=D}, yielding the desired inequality~\eqref{eq:miniBisect}.
The proof finishes by noting that if $M_+$ is vertical, then $K\subset R(a,+\infty)$, which gives the same inequality~\eqref{eq:2}.
Equality above only holds when $A(b)$ is maximum, namely for $R(1/\sqrt{5},+\infty)$, which coincides with $Q$ by definition.
%Now, let $(Q_1,Q_2)$ be a minimizing bisection of $Q$, determined by a curve $l:[-1,1]\rightarrow Q$.
%Note that the distance between the points $(0,\pm 2/\sqrt{5})$ is larger than $\D_B(Q)$,
%and hence these point belong to $Q_1$ and $Q_2$ respectively.
%Therefore, $(\pm 1/\sqrt{5},0)$ must be the endpoints of $l$.
%Moreover,
%\[
%l([-1,1])\subset \left\{(x_1,x_2)\in\mathbb R^2:x_1^2+\left(x_2\pm \frac{2}{\sqrt{5}}\right)\leq 1\right\},
%\]
%as desired.
\end{proof}

\begin{remark}
\label{re:futuro}
Let $(Q_1,Q_2)$ be a minimizing bisection of $Q$ determined by a curve $l:[-1,1]\to Q$,
and let $q_+=(0,2/\sqrt{5})\in\partial Q$, $q_-=(0,-2/\sqrt{5})\in\partial Q$,
$p_+=(1/\sqrt{5},0)\in\partial Q$, $p_-=(-1/\sqrt{5},0)\in\partial Q$.
Recall that $\D_B(Q)=1=d(p_-,q_+)$.
Since $d(q_+,q_-)>\D_B(Q)$,
each of these two points must belong to a different subset of the bisection.
We can assume that $q_+\in Q_1$, $q_-\in Q_2$.
Then, it necessarily follows that
$Q_1\subseteq B(q_+,D_B(Q))$ and $Q_2\subseteq B(q_-,D_B(Q))$,
where $B(x,r)$ denotes the closed ball centered at $x$ and of radius $r$.
Those inclusions immediately imply that $\{l(-1),l(1)\}=\{p_-,p_+\}$, and
also that $l([-1,1])$ is contained in the intersection of those balls, that is,
%since the only points of $\partial Q$ in both subsets of the bisection are precisely the endpoints of $l$.
%Finally, any point from $l([-1,1])$ has to be at most at distance $\D_B(Q)$ from $q_+$ (and also from $q_-$),
%which gives that
\[
l([-1,1])\subset\left\{(x_1,x_2)\in\mathbb R^2:x_1^2+\left(x_2\pm\frac{2}{\sqrt{5}}\right)^2\leq 1\right\}.
\]
\end{remark}

\begin{remark}
The reader will realize that the line segment $[(-a,0),(a,0)]$
does not give a minimizing bisection of $R(a,m)$ in the proof of Theorem~\ref{th:isodiametricIneq} \emph{for some values of the parameters} $a,m$.
Indeed, in every step of the proof of Theorem \ref{th:isodiametricIneq},
we replace the set by another one with greater (or equal)
%without smaller
area. This process starts with $K$ and ends with $Q=R(1/\sqrt{5},+\infty)$,
and the corresponding horizontal line segment provides a minimizing bisection \emph{for these two sets}, whereas in the middle of the
process, that line segment \emph{does not give} necessarily a minimizing bisection of $R(a,m)$ in general.
For instance, for $K=R(a,\sqrt{3})$, with $\D_B(K)>2a$, the bisection determined by the line segment $[(-a,0),(a,0)]$
is not minimizing, since it can be improved by a different line segment (placed slightly above).
%one of the subsets is always an equilateral triangle (with diameter equal to $2a$), whereas the other one is a curved pentagon
% (with diameter equal to $D_B(K)$) does not satisfy Lemma \ref{lem:EqualValueBisection}
\end{remark}

\section{The isominwidth inequality}\label{sec:IsominwidthIneq}
In this section we will consider the problem analogous to the one studied in Section~\ref{sec:IsodiamIneq},
but for the width functional.
We will start by proving that $\w_B(K)=\w(K)/2$, for any $K\in\K^2$,
by using the following celebrated result by Bang on Tarski's plank problem \cite{Ba}:
for $K\in\K^2$, and $p,q\in\mathrm{bd}(K)$, let $(K_1,K_2)$ be the bisection
given by the line segment $[p,q]$. Then
\begin{equation}\label{eq:Bang}
\w(K_1)+\w(K_2)\geq\w(K).
\end{equation}

\begin{lemma}\label{lem:WidthBisect=Half}
Let $K\in\K^2$. Then, $\w_B(K)=\w(K)/2$.
\end{lemma}

\begin{proof}
Let $L_1$, $L_2$ be two parallel supporting lines of $K$ such that $d(L_1,L_2)=\w(K)$,
and let $u\in\mathbb S^1$ be an orthogonal vector to these lines.
Consider $p,q\in\mathrm{bd}(K)$ such that $[p,q]=K\cap L$, where $L$ is a line parallel to $L_i$ and lies at distance $\w(K)/2$
from each line $L_i$, $i=1,2$. Moreover, let $(K_1,K_2)$ be the bisection determined by the line segment $[p,q]$.
Note that $L$ and $L_i$ are supporting lines of $K_i$, for $i=1,2$, and so $\w(K_i)\leq\w(K)/2$.
Thus, $\max\{\w(K_1),\w(K_2)\}\leq\w(K)/2$ and hence $\w_B(K)\leq\w(K)/2$.
On the other hand, in view of Lemmas~\ref{lem:StraightLine},~\ref{lem:minimum} and~\ref{lem:EqualValueBisection},
let $(\wt{K_1},\wt{K_2})$ be a bisection of $K$ where $\w_B(K)$ is attained, given by a line segment
and satisfying $\w_B(K)=\w(\wt{K_1})=\w(\wt{K_2})$.
Then, \eqref{eq:Bang} implies that
$\w(K)\leq\w(\wt{K_1})+\w(\wt{K_2})=2\,\w_B(K)$, and so $\w_B(K)\geq\w(K)/2$, yielding the desired equality.
\end{proof}

Now we are able to prove immediately the main result of this section,
which is Theorem~\ref{th:isominwidthIneq}, providing a sharp upper bound for $\w_B$.

\begin{proof}[Proof of Theorem \ref{th:isominwidthIneq}]
By Lemma \ref{lem:WidthBisect=Half} and Pal's inequality \eqref{eq:palIneq} we directly have that
\[
\frac{\A(K)}{\w_B(K)^2}=4\,\frac{\A(K)}{\w(K)^2}\geq\frac{4}{\sqrt{3}}.
\]
Moreover, in order to have equality, we must have equality in \eqref{eq:palIneq},
hence implying that $K$ is an equilateral triangle $\mathcal{T}$.
Additionally, %in order to find a bisection of $\mathcal{T}$ attaining $\w_B(\mathcal{T})$,
note that $\w(\mathcal{T})$ coincides with any of the three heights of $\mathcal{T}$,
and the width corresponding to any other different direction will be strictly greater.
Since $\w_B(\mathcal{T})=\w_B(\mathcal{T})/2$ by Lemma~\ref{lem:WidthBisect=Half},
this implies that any minimizing bisection $(\mathcal{T}_1, \mathcal{T}_2)$ of $\mathcal{T}$
must satisfy that $\mathcal{T}_1\cap\mathcal{T}_2$ is a line segment whose endpoints are
the midpoints of two edges of $\mathcal{T}$.
%Since $\w(\mathcal{T})$ is only attained in any of the three heights of $\mathcal{T}$,
%and Lemma \ref{lem:WidthBisect=Half} implies that $\w_B(\mathcal{T})=\w(\mathcal{T})/2$, then a minimizing bisection $(\mathcal{T}_1,\mathcal{T}_2)$ of $\mathcal{T}$
%is exactly given in those same directions, each of them having in each direction $\w(\mathcal{T}_i)=\w_B(\mathcal{T})$, $i=1,2$, and thus
%having that $\mathcal{T}_1\cap \mathcal{T}_2$ is the line segment with endpoints the midpoints of two edges of $\mathcal{T}$.
\end{proof}

\section{The Behrend-Bisecting position and the reverse isodiametric inequality}\label{sec:RevIsodiamIneq}

As commented in the Introduction, we will now focus on a reverse isodiametric-type inequality for $D_B$.
The following definitions and results arise mainly from some ideas in~\cite{Beh}.
For every $K\in\K^2$, let $$V_K:=\{u\in\mathbb S^1:\exists\,x\in K\,\text{ such that }x+\D(K)[0,u]\subset K\}$$
% if $u\in V_k$, then $\D(K)[0,u]$ is a segment of length $\D(K)$ placed in the direction of $u$. The effect of the part "x+" is that we place that segment in the point $x$. If this is contained in $K$, it means that it is indeed a segment of length $\D(K)$ contained in $K$, that is, $\D(K)$ is attained in that segment (and $u$ is a direction where that value is attained)
be the set of \emph{diametrical directions} of $K$ (that is, the directions for which $\D(K)$ is attained).
Moreover, we will say that $u\in\mathbb{S}^1$ is a \emph{bisector of $K$}
if $u$ is the direction of a line segment providing a minimizing bisection $(K_1,K_2)$ of $K$ with $\D(K_1)=\D(K_2)$.
We will denote by $B_K$ the set of bisectors of $K$.
Note that $B_K$ contains the directions which determine suitable minimizing bisections by line segments for $\D_B$.

The next result establishes that the supremum in the definition of the Behrend-bisecting position \eqref{def:BBP} is
actually a maximum.

\begin{lemma}\label{lem:isodiamMaximum}
Let $K\in\K^2$. Then, there exists $\phi\in End(\mathbb R^{2})$ such that $\phi(K)$ is in
Behrend-bisecting position.
\end{lemma}

\begin{proof}
We can assume, after a suitable translation of $K$,
that $r\,\mathbb B^2_2\subseteq K$ for some $r>0$,
where $\mathbb B^2_2$ is the \emph{planar Euclidean unit ball centered at the origin}.
Call
\[
\rho:=\sup_{\phi\in End(\mathbb R^{2})}\frac{\A(\phi(K))}{\D_B(\phi(K))^2}.
\]
Since $\A$ and $\D_B^2$ are homogeneous functionals of degree two, we can suppose without loss of generality that
$|\det(\phi)|=1$, $\A(K)=1$, and
\begin{equation}
\label{eq:rho}
\inf_{\substack{\phi\in End(\mathbb R^{2}) \\[0.5mm] |\det(\phi)|=1}}\D_B(\phi(K))=\frac{1}{\sqrt{\rho}}.
\end{equation}
Consider a sequence $\{\phi_i\}_{i\in\mathbb N}\subset End(\mathbb R^{2})$
such that $|det(\phi_i)|=1$, for $i\in\mathbb N$, and $$\D_B(\phi_i(K))\rightarrow\frac{1}{\sqrt{\rho}} \text{ when } i\rightarrow\infty.$$
We can additionally assume that all the endomorphisms $\phi_i$ are \emph{linear},
since $\D_B$ is invariant under translations, and that there exists $C>0$ such that $\D_B(\phi_i(K))\leq C$ for every $i\in\mathbb N$.
Since $(0,0)\in\phi_i(K)$ and $\D(\phi_i(K))\leq 2\,\D_B(\phi_i(K))\leq 2\,C$, for all $i\in\mathbb N$,
then $\{\phi_i(K)\}_{i\in\mathbb N}$ is an absolutely bounded sequence
(since $(0,0)\in\phi_i(K)$, we actually have that $\phi_i(K)\subseteq 2\,C\,\mathbb B^2_2$).
Hence the Blaschke Selection Theorem implies that there exists a subsequence
(which will be denoted as the original one) such that $\phi_i(K)\rightarrow K_0$
when $i\rightarrow\infty$, for some $K_0\in\K^2$. Let us furthermore observe that
if $\phi_i=(a^i_{jk})_{1\leq j,k\leq 2}\in\mathbb R^{2\times 2}$,
since $r\,\mathbb B^2_2\subseteq K$ and $\phi_i(K)\subseteq 2\,C\,\mathbb B^2_2$, then
it follows that $|a^i_{jk}|\leq 2\,C/r$ for $1\leq j,k\leq 2$ and $i\in\mathbb N$.
Thus $\{\phi_i\}_{i\in\mathbb N}$ is bounded
with respect to the so-called induced norm (or operator norm) $\Vert\cdot\Vert_{op}$ for linear endomorphisms,
and so there exists a subsequence (which will be denoted again as the original one)
such that $\phi_i\rightarrow\phi_0$ when $i\rightarrow\infty$, for some $\phi_0\in End(\mathbb R^{2})$.
Moreover, $|\det(\phi_0)|=1$, with
$\phi_i(K)\rightarrow K_0=\phi_0(K)$ when $i\rightarrow\infty$.
We will now prove that $\D_B(\phi_0(K))=1/\sqrt{\rho}$, which will imply that $\phi_0(K)$ is in Behrend-bisecting position, as desired.

First of all, since each $\phi_i$ is linear and regular, we have that $\phi_i$ is bijective.
Fix $u_i\in B_{\phi_i(K)}$, and let $x_i\in K$, $\mu_i>0$ be such that
the line segment $\phi_i(x_i)+\mu_i[0,u_i]\subset\phi_i(K)$ provides a minimizing bisection of $\phi_i(K)$, for each $i\in\mathbb N$.
Let $\phi(K^i_1),\, \phi(K^i_2)$ be the subsets of that bisection, satisfying
$\D_B(\phi_i(K))=\D(\phi_i(K^i_1))=\D(\phi_i(K^i_2))$ for every $i\in\mathbb N$.
Since $\phi_i$ is a bijection, we will have that $(K^i_1,K^i_2)$ is a bisection of $K$
and moreover, we can consider $y_i\in K$ such that $\phi_i(y_i)=\phi_i(x_i)+\mu_i\,u_i$, for every $i\in\mathbb N$.
Since $\{[x_i,y_i]\}_{i\in\mathbb N}\subset K$ is again absolutely bounded, we can suppose
that $[x_i,y_i]\rightarrow[x_0,y_0]$ when $i\rightarrow\infty$,
for some $x_0$, $y_0\in K$.
Let $(K^0_1, K^0_2)$ be the bisection of $K$ given by $[x_0,y_0]$,
and let us see that
%Since $K^i_j\rightarrow K^0_j$ when $i\rightarrow\infty$,
$\phi_i(K^i_j)\rightarrow\phi_0(K^0_j)$ when $i\rightarrow\infty$, for $j=1,2$.
By the subadditivity of Hausdorff distance $d_{\mathcal H}$, it is clear that
\begin{equation}
\label{triangular}
d_{\mathcal H}(\phi_i(K_j^i),\phi_0(K^0_j))\leq d_{\mathcal H}(\phi_i(K_j^i),\phi_0(K^i_j))+d_{\mathcal H}(\phi_0(K_j^i),\phi_0(K^0_j)).
\end{equation}
Note that, since $K$ is compact, then $K\subset\delta\,\mathbb B^2_2$, for some $\delta>0$,
and thus $K_j^i\subset \delta\,\mathbb B^2_2$ for every $i\in\mathbb N$ and $j=1,2$.
Then, for $x\in K_j^i$,
%\[
%\Vert\phi_i(x)-\phi_0(x)\Vert_2 \leq \sqrt{2}\, \Vert\phi_i(x)-\phi_0(x)\Vert_\infty \leq\, \sqrt{2}\, \Vert\phi_i-\phi_0\Vert_\infty\Vert x\Vert_\infty \leq \sqrt{2}\,
%\Vert\phi_i-\phi_0\Vert_\infty \delta,
%\]
\[
\Vert\phi_i(x)-\phi_0(x)\Vert_2 =\, \Vert(\phi_i-\phi_0)(x)\Vert_2\leq\,\Vert\phi_i-\phi_0\Vert_{op}\,\Vert x\Vert_2 \leq \delta\,\Vert\phi_i-\phi_0\Vert_{op},
\]
which implies that $d_{\mathcal H}(\phi_i(K_j^i),\phi_0(K^i_j))\to0$ when $i\to\infty$, for $j=1,2$.
On the other hand, we claim that
$d_{\mathcal H}(\phi_0(K_j^i),\phi_0(K^0_j))\leq\Vert \phi_0\Vert_{op}\,d_{\mathcal H}(K_j^i,K_j^0)$.
Consider $\varepsilon_j^i:=d_{\mathcal H}(K_j^i,K_j^0)$,
which tends to $0$ when $i\to\infty$, for $j=1,2$.
It follows that
$K_j^i\subseteq K_j^0+\varepsilon_j^i\,\mathbb B^2_2$,
and by applying $\phi_0$ we get
$\phi_0(K_j^i)\subseteq \phi_0(K_j^0)+\varepsilon_j^i\,\Vert \phi_0\Vert_{op}\,\mathbb{B}^2_2$.
Analogously,
we will get
$\phi_0(K_j^0)\subseteq \phi_0(K_j^i)+\varepsilon_j^i\,\Vert \phi_0\Vert_{op}\,\mathbb{B}^2_2$.
These two inclusions yield the claim, by the definition of $d_{\mathcal H}$, which implies that
$d_{\mathcal H}(\phi_0(K_j^i),\phi_0(K^0_j))\to0$ when $i\to\infty$, for $j=1,2$.
Taking into account \eqref{triangular}, we conclude that
$\phi_i(K^i_j)\rightarrow\phi_0(K^0_j)$ when $i\rightarrow\infty$, for $j=1,2$.
%
%Therefore
%\[
%d_{\mathcal H}(\phi_i(K_j^i),\phi_0(K^0_j))\leq d_{\mathcal H}(\phi_i(K_j^i),\phi_0(K^i_j)) + \Vert \phi_0\Vert_\infty d_{\mathcal H}(K_j^i,K_j^0).
%\]
%It suffices to show that $d_{\mathcal H}(\phi_i(K_i^j),\phi_0(K^j_i))$ tends to $0$ when $i$ tends to infinity.
%Since $K$ is compact, then $K\subset\rho\mathbb B^2_2$, for some $\rho>0$, and thus $K_j^i\subset \rho\mathbb B^2_2$ for every $i\geq 1$ and $j=1,2$.
%Notice that for every $x\in\rho\mathbb B^2_2$, then
%\[
%\textcolor{red}{\Vert\phi_i(x)-\phi_0(x)\Vert_2 \leq  \Vert\phi_i(x)-\phi_0(x)\Vert_\infty \leq \Vert\phi_i-\phi_0\Vert_\infty\Vert x\Vert_\infty \leq
%\Vert\phi_i-\phi_0\Vert_\infty \rho,}
%\]
%\textcolor{red}{and thus we can conclude that $d_{\mathcal H}(\phi_i(K_j^i),\phi_0(K^0_j))\rightarrow 0$ if $i\rightarrow\infty$, $j=1,2$.}
Therefore $\D(\phi_0(K^0_j))=1/\sqrt{\rho}$, for $j=1,2$,
and so $\D_B(K_0)\leq 1/\sqrt{\rho}$. But if this inequality is strict, we get a contradiction with \eqref{eq:rho},
so equality must hold, which finishes the proof.
\end{proof}

The proof of the following characterization of the Behrend position for a convex body can be found in~\cite{GMS}
(equivalence (ii) was already proved by Behrend~\cite{Beh}).

\begin{proposition}\label{prop:behrendPos}
Let $K\in\K^2$. The following statements are equivalent.
\begin{enumerate}
\item[(i)] $K$ is in Behrend position.
\item[(ii)] For every $u\in\mathbb S^1$, there exists $v\in V_K$ such that $|u^Tv|\geq 1/\sqrt{2}$.
\item[(ii')] For every $u\in\mathbb S^1$, there exists $v\in V_K$ such that $|u^Tv|\leq 1/\sqrt{2}$.
\item[(iii)] There exist $u_i\in V_K$ and $\lambda_i\geq 0$, $i=1,2,3$, such that $\displaystyle{\sum_{i=1}^3\lambda_i(u_i\,u_i^T)=\mathrm{I}_2}$,
where $\mathrm{I}_2$ denotes the identity matrix of degree two.
\end{enumerate}
\end{proposition}

\begin{remark}
Condition (ii) (resp.,~(ii')) in Proposition~\ref{prop:behrendPos} means that for any fixed $u\in\mathbb S^1$,
there exists a diametrical direction $v\in V_K$ contained in the double cone (resp.,~outside the double cone) with apex at $0$
and vectors making an angle of at most $\pi/4$ radians with respect to $\pm u$.
Condition (iii) states that the identity matrix of degree two admits a decomposition
as a non-negative linear combination of matrices of rank one,
by means of three certain diametrical directions of $K$
(cf.~\cite{GMS} and the references therein for further details and connections with other results).
\end{remark}

Next result establishes the analogous in Proposition \ref{prop:behrendPos} to (i) implies (ii) or (iii).
The proof is inspired in the ideas from~\cite[Lemma 3.2]{GMS}.

\begin{lemma}\label{lem:behrendBisectPos}
Let $K\in\K^2$ be in Behrend-bisecting position.
For every $u\in\mathbb S^1$ and every $w\in B_K$, being $(K_1^w, K_2^w)$ the corresponding minimizing bisection of $K$,
we have that
\begin{enumerate}
\item[(i)] there exists $v\in V_{K_1^w}\cup V_{K_2^w}$ such that $|u^Tv|\geq 1/\sqrt{2}$, and
\item[(ii)] there exists $v\in V_{K_1^w}\cup V_{K_2^w}$ such that $|u^Tv|\leq 1/\sqrt{2}$.
\end{enumerate}
\end{lemma}

\begin{proof}
We start by proving (i). Let us suppose that for every $v\in V_{K^w_1}\cup V_{K^w_2}$
we have that $|u^Tv|<1/\sqrt{2}$. Hence every $v\in V_{K^w_1}\cup V_{K^w_2}$ makes an angle $\theta$
with the line $u^\bot$ satisfying
\[
\theta=\frac{\pi}{2}-\arccos(u^Tv)=\arcsin(u^Tv)<\arcsin\frac{1}{\sqrt{2}}=\frac{\pi}{4},
\]
and so $\cos^2\theta>1/2$.
% From the hypothesis $|u^t v|<1/sqrt{2}$, we get that $cos(u^t,v)>1/sqrt{2}$ and $\angle(u^t,v)$ is in $[45,135]$, and then, for the supplementary angle $u^\bot=\pi/2-u^t$, we will get $\cos(u^\bot,v)=\sin(u^t,v)>1\sqrt{2}$.
More precisly, since $K$ is compact (as well as $K_i^w$, for $i=1,2$), there exists $\delta>0$ such that
for every $v\in V_{K_1^w}\cup V_{K_2^w}$ making angle $\theta$ with respect to $u^\bot$, we have
\begin{equation}\label{eq:Proof1}
\cos^2\theta>\frac12(1+\delta).
\end{equation}
After a suitable rotation of $K$, we can suppose that $u=(1,0)$. For small $\varepsilon>0$, consider the endomorphism of $\mathbb R^2$
determined by the matrix
\[
A_\varepsilon:=\left(\begin{array}{cc} 1 & 0 \\ 0 & 1-\varepsilon \end{array}\right).
\]
Using elementary trigonometry and calculus, we can see that the length of any line segment $\ell$, making angle
$\theta$ with $u^\bot$, varies under $A_\varepsilon$ according to the formula
\begin{equation}\label{eq:Proof2}
||A_\varepsilon\ell||=||\ell||\sqrt{1-2\,\varepsilon\cos^2\theta+\varepsilon^2\cos^2\theta}
=||\ell||(1-\varepsilon\cos^2\theta+O(\varepsilon^2)).
\end{equation}

Let $K'=A_\varepsilon K$ and $(K^w_i)'=A_\varepsilon K^w_i$, for $i=1,2$
(since $A_\varepsilon$ is bijective, then $((K_1^w)',(K_2^w)')$ is a bisection of $K'$).
As $A_\epsilon$ is close to the identity matrix for small $\varepsilon$, and $K$, $K^w_1$, $K_2^w$ are compact sets,
for every $v'\in D_{(K_1^w)'}\cup D_{(K^w_2)'}$ making an angle $\theta'$ with $u^\bot$
it is possible to choose $\delta'>0$ small enough such that
%\[
%|u^Tv'|<\frac{1}{\sqrt{2}}\quad\text{  for every }v'\in D_{(K_1^w)'}\cup D_{(K^w_2)'}.
%\]
\begin{equation}
\label{eq:theta}
\cos^2\theta'>\frac{1}{2}(1+\delta').
\end{equation}

Let $A_\varepsilon(\ell)$ be the line segment in $K'$ with $||A_\varepsilon(\ell)||=\max\{\D((K_1^w)'),\D((K_2^w)')\}$,
being $\ell$ the corresponding line segment in $K$, making angle $\theta''$ with $u^\bot$.
Then, equation~\eqref{eq:theta} implies that there exists $\delta''$ such that
\[
\cos^2\theta''>\frac{1}{2}(1+\delta''),
\]
since $A_\varepsilon^{-1}$ is also close to the identity matrix.
Thus, taking into account \eqref{eq:Proof2} and the fact that $w\in B_K$, we have
%Thus, if $\ell$ is a line segment whose image under $A_\varepsilon$ attains $\D_B(K')$
%and which makes an angle $\theta$ with $u^\bot$, we get by for $i=1,2$ by \eqref{eq:Proof2}
%that
\begin{align*}\label{eq:Proof3}
\D_B(K')&\leq\max\{\D((K_1^w)'),\D((K_2^w)')\}=||A_\varepsilon(\ell)||\\
&=||\ell||(1-\varepsilon\cos^2\theta+O(\varepsilon^2))\leq\max\{\D(K^w_1),\D(K^w_2)\}\,(1-\varepsilon\cos^2\theta+O(\varepsilon^2))\\
& =\D_B(K)\,(1-\varepsilon\cos^2\theta+O(\varepsilon^2)),
\end{align*}
and so, since $\A(K')=\A(A_\varepsilon K)=(1-\varepsilon)\,\A(K)$, we conclude that %and then for small enough $\varepsilon$
\[
\begin{split}
\frac{\A(K')}{\D_B(K')^2} & \geq\frac{\A(K)}{\D_B(K)^2}\frac{1-\varepsilon}{(1-\varepsilon\cos^2\theta''+O(\varepsilon^2))^2} \\
& \geq \frac{\A(K)}{\D_B(K)^2}\frac{1-\varepsilon}{1-2\,\varepsilon\cos^2\theta''+O(\varepsilon^2)}\\
& >\frac{\A(K)}{\D_B(K)^2} \frac{1-\varepsilon}{1-(1+\delta'')\varepsilon+O(\varepsilon^2)}>\frac{\A(K)}{\D_B(K)^2},
\end{split}
\]
for $\varepsilon$ small enough, contradicting the fact that $K$ is in Behrend-bisecting position.

On the other hand, (ii) follows directly from (i),
since (ii) holds for $u\in\mathbb S^1$ if (i) holds for $u'\in\mathbb S^1\cap u^\bot$ (and viceversa).
\end{proof}

\begin{remark}
\label{re:triangles}
We will now see that, in contrast with Proposition~\ref{prop:behrendPos},
the necessary condition in Lemma~\ref{lem:behrendBisectPos}
for $K$ to be in Behrend-bisecting position is not \emph{sufficient}.
Let $K^\theta\in\K^2$ be the \emph{isosceles triangle with different angle $\theta\in[0,\pi/3]$}, with
$p_1$ the vertex of angle $\theta$, and $p_2,p_3$ the other two vertices.
For any minimizing bisection $(K^\theta_1,K^\theta_2)$ of $K^\theta$ determined by a line segment,
we can assume that $p_1\in K^\theta_1$ and $p_2,p_3\in K^\theta_2$
(otherwise, the diameter of one of the subsets will be equal to $\D(K^\theta)$, and so the bisection will not be minimizing).
By a suitable rescaling, we can suppose without loss of generality that
$p_2=(1,0)$, $p_3=(-1,0)$, and $p_1=(0,\tan((\pi-\theta)/2))$.
The distance from $q_\lambda=(1-\lambda)\,p_1+\lambda\,p_2$
to $p_1$ equals $\lambda\sqrt{1+\tan((\pi-\theta)/2)^2}$,
whereas to $p_3$ equals $\sqrt{(1+\lambda)^2+(1-\lambda)^2\tan((\pi-\theta)/2)^2}$.
\begin{figure}[ht]
  \includegraphics[width=0.22\textwidth]{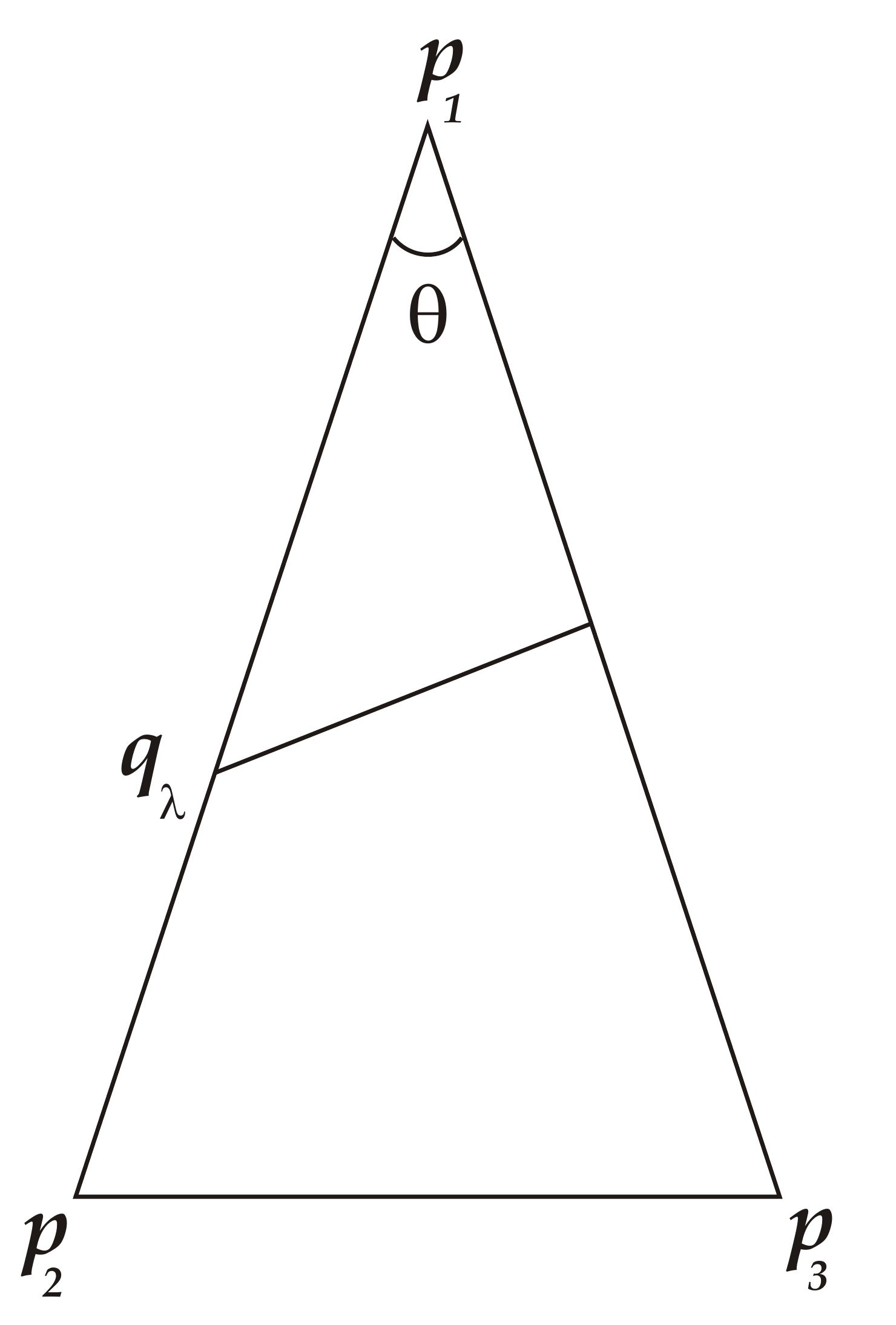}\\
  \caption{An isosceles triangle $K^\theta$ and an arbitrary bisection of $K^\theta$}\label{fig:isosceles}
\end{figure}
Since the bisection is minimizing, these two distances must coincide, and so
the value of $\lambda$ must be equal to
\[
\lambda_m=\lambda_m(\theta)=\frac{1+\tan(\frac{\pi-\theta}{2})^2}{2\left(\tan(\frac{\pi-\theta}{2})^2-1\right)}.
\]
An analogous reasoning for the points of the edge $\overline{p_1\,p_3}$
yields that the \emph{only} minimizing bisection by a line segment is given by the horizontal segment
\[
\bigg[\bigg(-\lambda_m,(1-\lambda_m)\tan\bigg(\frac{\pi-\theta}{2}\bigg)\bigg),\bigg(\lambda_m,(1-\lambda_m)\tan\bigg(\frac{\pi-\theta}{2}\bigg)\bigg)\bigg].
\]
In this case,
\[
\lambda_m\bigg(\pm 1,-\tan\bigg(\frac{\pi-\theta}{2}\bigg)\bigg) \in V_{K^\theta_1}\quad\text{and}\quad
\bigg(\pm (\lambda_m+1),(1-\lambda_m)\tan\bigg(\frac{\pi-\theta}{2}\bigg)\bigg)\in V_{K^\theta_2}.
\]
It can be checked that for $\theta\in[\pi/6,\pi/3]$, the triangles $K^\theta$ satisfy the thesis in Lemma~\ref{lem:behrendBisectPos},
by a direct analysis of the positions of the vectors of $V_{K^\theta_1}\cup V_{K^\theta_2}$.
However, not all of those triangles are in Behrend-bisecting position. Note that the isodiametric quotient
\[
\frac{\A(K^\theta)}{\D_B(K^\theta)^2}=\frac{\tan(\frac{\pi-\theta}{2})}{\lambda_m^2(1+\tan(\frac{\pi-\theta}{2})^2)}=2\cos^2(\theta)\sin(\theta)
\]
attains its maximum value in the interval $[0,\pi/3]$ only when $\theta=\theta_M=\arccos(\sqrt{2/3})$   %=\pi-2\arctan(\sqrt{2}+\sqrt{3})$
($\approx 35.26^\circ$), with maximum value
\[
\frac{\A(K^{\theta_M})}{\D_B(K^{\theta_M})^2}=\frac{4}{3\sqrt{3}},
\]
which implies that $K^{\theta_M}$ is the only isosceles triangle $K^\theta$, with $\theta\in[0,\pi/3]$,
which is a candidate for being in Behrend-bisecting position.
\end{remark}

The following Proposition~\ref{prop:TriangleOptimal} proves that, in fact,
the unique \emph{triangle} in Behrend-bisecting position is $K^{\theta_M}$ from Remark~\ref{re:triangles}.

\begin{proposition}\label{prop:TriangleOptimal}
\label{prop:triangles}
The unique triangle in Behrend-bisecting position is $K^{\theta_M}$ from Remark~\ref{re:triangles}.
\end{proposition}

\begin{proof}
First of all, we will see that, in the class of isosceles triangles,
the isodiametric quotient is uniquely maximized by $K^{\theta_M}$.
Let $K^\theta\in\K^2$ be now the isosceles triangle with different (largest) angle $\theta\in[\pi/3,\pi]$.
Let $p_1$ be the vertex of angle $\theta$, and let $p_2,p_3$ be the other two vertices.
For any minimizing bisection $(K^\theta_1,K^\theta_2)$ of $K^\theta$ determined by a line segment,
we can now suppose that $p_1,p_2\in K^\theta_1$ and that $p_3\in K^\theta_2$,
and so $d(p_1,p_2)\leq\D_B(K^\theta)$.
In particular, if we consider the bisection given by the line segment $[p_1,(1/2)(p_2+p_3)]$,
then $\D(K^\theta_1)=\D(K^\theta_2)=d(p_1,p_2)$, and so $\D_B(K^\theta)=d(p_1,p_2)$.
Call $a=d(p_1,p_2)$ and $b=d(p_2,p_3)$. Then, basic computations show that $b=2\,a\sin(\theta/2)$ and
\[
\frac{\A(K^\theta)}{\D_B(K^\theta)^2}=\frac{\frac{1}{2}(2\,a\sin(\frac{\theta}{2}))\sqrt{a^2-a^2\sin(\frac{\theta}{2})^2}}{a^2}
=\sin\bigg(\frac{\theta}{2}\bigg)\cos\bigg(\frac{\theta}{2}\bigg)=\frac{\sin \theta}{2},
\]
and hence
\[
\frac{\A(K^\theta)}{\D_B(K^\theta)^2}\leq\frac{\A(K^{\frac{\pi}{2}})}{\D_B(K^{\frac{\pi}{2}})^2}=\frac{1}{2}\leq\frac{\A(K^{\theta_M})}{\D_B(K^{\theta_M})^2},
\]
taking into account Remark~\ref{re:triangles}.

Now consider a general triangle $K\in\K^2$.
We can assume that $K=\conv(\{p_1,p_2,p_3\})$ for some $p_i\in\mathbb R^2$, $i=1,2,3$,
with $\D(K)=d(p_1,p_2)$. %, and hence let $\w(K)$ be the distance between $p_3$ and $[p_1,p_2]$ (the \emph{height} of $K$).
Let $\alpha_i>0$ be the angle at vertex $p_i$, for $i=1,2,3$, with $\alpha_1\leq\alpha_2\leq\alpha_3$.
For any minimizing bisection $(K_1,K_2)$ of $K$, we can suppose that $p_1\in K_1$ and that $p_2$, $p_3\in K_2$
(otherwise, the bisection will not be minimizing).
Call $q_\lambda=(1-\lambda)\,p_1+\lambda\,p_3$, and let $\lambda_m\in[0,1]$ be such that
the distance $d_1$ from $q_{\lambda_m}$ to $p_1$ is the same than to $p_2$. Analogously, consider
$r_\mu:=(1-\mu)\,p_1+\mu\,p_2$, and let $\mu_m\in[0,1]$ be such that the distance $d_2$ from $r_{\mu_m}$ to $p_1$ is the same than to
$p_3$.
\begin{figure}[ht]
  \includegraphics[width=0.34\textwidth]{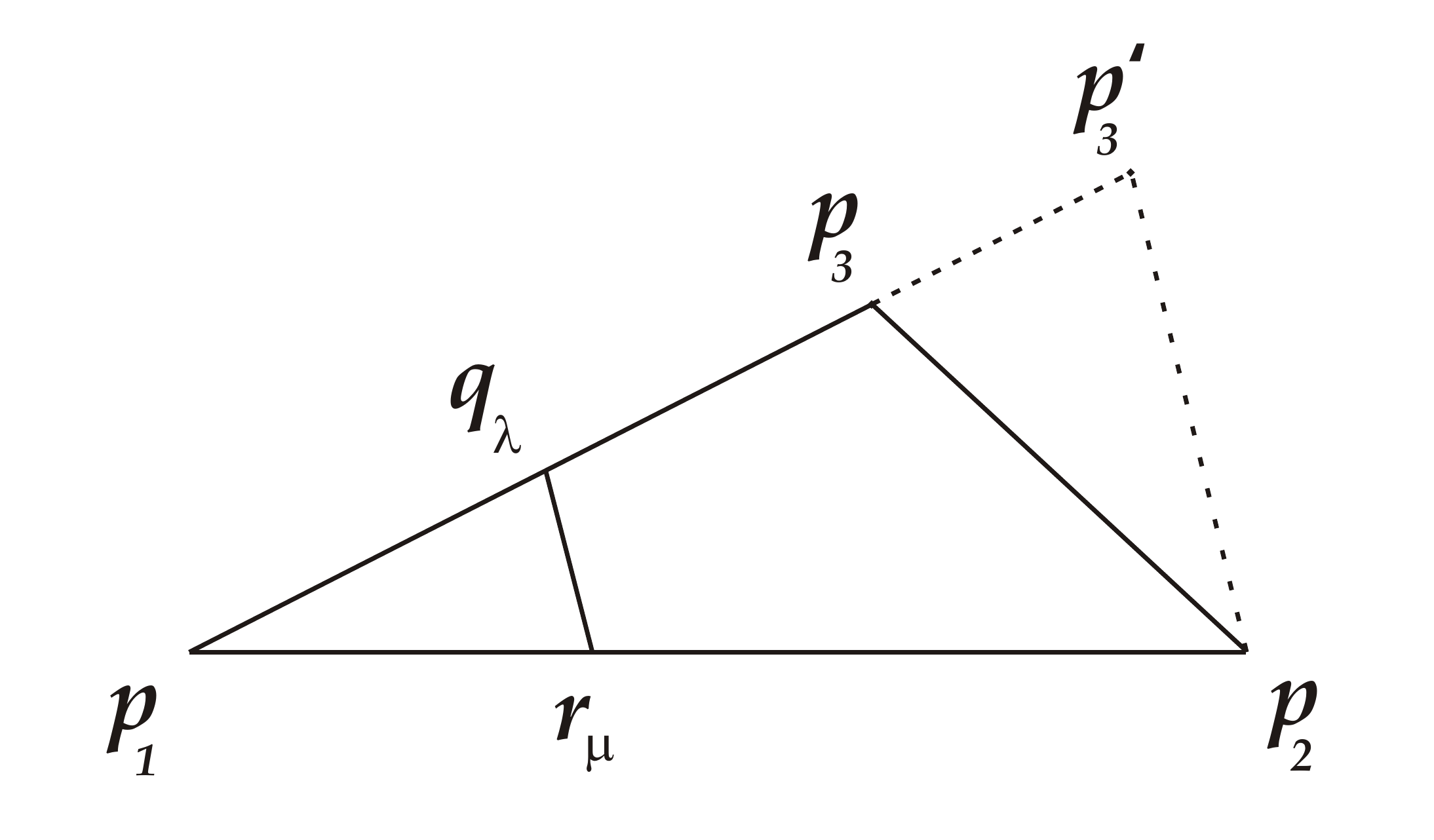}\\
  \caption{A general triangle $K$ and its extended isosceles triangle $K'$}\label{fig:extended}
\end{figure}
In this case, and since the distance from $p_1$ to $p_3$ is not larger than to $p_2$,
we clearly have that $d_1\geq d_2$, and hence the line segment with endpoints $q_{\lambda_m}$ and $r_{\mu_m}$ provides a minimizing bisection of $K$,
with subsets $K_1=\conv(\{p_1,q_{\lambda_m},r_{\mu_m}\})$ and $K_2=\conv(\{p_2,p_3,q_{\lambda_m},r_{\mu_m}\})$ satisfying that
$\D(K_1)=\D(K_2)=d_1$.
%\[
%K_1=\conv\{p_1,q_{\lambda_m},r_{\mu_m}\},\quad K_2=\conv\{p_2,p_3,q_{\lambda_m},r_{\mu_m}\}
%\]
Let $p_3'$ be the point in the ray from $p_1$ to $p_3$ which is at the same
distance from $p_1$ than from $p_2$,
and consider the isosceles triangle $K'=\conv(\{p_1,p_2,p_3'\})$.
Then we clearly have that $K\subseteq K'$.
Moreover, a bisection attaining $\D_B(K')$ is given again by the line segment with endpoints $q_{\lambda_m}$
and $(1-\lambda_m)\,p_1+\lambda_m\,p_2$, with $\D_B(K')=\D_B(K)=d_1$. Hence
\[
\frac{\A(K)}{\D_B(K)^2}\leq\frac{\A(K')}{\D_B(K')^2},
\]
which implies that the isodiametric quotient of $K$ is always maximized by the isodiametric quotient
of an isosceles triangle whose different angle is not larger than $\pi/3$ (because $\alpha_1\leq\pi/3$).
Taking into account Remark~\ref{re:triangles} (and the fact that any planar triangle can be obtained by applying
an appropriate affine endomorphism to $K$), we conclude that the unique triangle in Behrend-bisecting position
is the isosceles triangle $K^{\theta_M}$ from Remark~\ref{re:triangles}.
\end{proof}

In view of Proposition~\ref{prop:triangles}, and taking into account the results from \cite{Beh},
it is natural to conjecture the following \emph{optimal} reverse isodiametric-type inequality for bisections.

\begin{conjecture}
Let $K\in\K^2$ be in Behrend-bisecting position. Then,
\[
\frac{\A(K)}{\D_B(K)^2}\geq\frac{4}{3\sqrt{3}},
\]
with equality if and only if $K$ is the isosceles triangle with different angle equal to $\arccos(\sqrt{2/3})$.
\end{conjecture}

The following proof is strongly inspired in the original proof of Behrend \cite{Beh} for showing \eqref{eq:BehrendIneq}.

\begin{corollary}\label{cor:angles}
Let $K\in\K^2$ be in Behrend-bisecting position.
Given $w\in B_K$, let $(K^w_1,K^w_2)$ be the corresponding minimizing bisection of $K$.
Then, there exist $u_1,\,u_2\in V_{K^w_1}\cup V_{K^w_2}$ such that $|u_1^Tu_2|\leq 1/2$.
\end{corollary}

\begin{proof}
Call $e_1=(1,0)$ and $e_2=(0,1)$.
By applying a proper rotation,
we can assume that $e_1\in V_{K^w_1}\cup V_{K_2^w}$.
Then, for $e_2\in\mathbb{S}^1$, by Lemma~\ref{lem:behrendBisectPos}\,(i),
there exists $u=(\cos\alpha,\sin\alpha)\in V_{K^w_1}\cup V_{K_2^w}$ such that $|e_2^Tu|\geq 1/\sqrt{2}$,
which implies that $\alpha\in[\pi/4,3\pi/4]$.
We can assume that $\alpha\in[\pi/4,\pi/2]$, by reflecting $K$ with respect to $\lin(e_2)$ if necessary.
If $\alpha\geq\pi/3$, then $|e_1^Tu|\leq 1/{2}$, which proves the statement for $u_1=e_1$ and $u_2=u$.
So assume that $\alpha<\pi/3$, and note that, taking into account the previous argument,
we can suppose that $(\cos\mu,\sin\mu)\notin V_{K^w_1}\cup V_{K^w_2}$ for $\mu\in[\pi/3,2\pi/3]$.
Consider the vector $\wt{u}=(\cos(\pi/3+\pi/4),\sin(\pi/3+\pi/4))\in\mathbb S^1$.
Again by Lemma~\ref{lem:behrendBisectPos}\,(i),
there exists $v=(\cos\beta,\sin\beta)\in V_{K^w_1}\cup V_{K^w_2}$ such that $|\wt{u}^Tv|\geq 1/\sqrt{2}$.
This necessarily implies that $2\pi/3<\beta\leq \pi/3+\pi/2=5\pi/6<\pi$. In
particular, the angle between $u$ and $v$ is at least $2\pi/3-\pi/3=\pi/3$
and at most $5\pi/6-\pi/4=7\pi/12<2\pi/3$, and thus we have that $|u^Tv|\leq 1/2$, as desired
(just take $u_1=u$ and $u_2=v$).
\end{proof}

We are now able to prove Theorem~\ref{th:RevIsodiamBisectIneq}.

\begin{proof}[Proof of Theorem \ref{th:RevIsodiamBisectIneq}]
Since $K$ is in Behrend-bisecting position, for a given $w\in B_K$, there exist
$u_1,\,u_2\in V_{K^w_1}\cup V_{K^w_2}$ such that $|u_1^Tu_2|\leq 1/2$, by Corollary \ref{cor:angles},
where $(K^w_1,K^w_2)$ is the corresponding minimizing bisection.
Since $\D(K^w_1)=\D(K^w_2)=\D_B(K)$, there exist $x_1,x_2\in K$
such that $x_1+\D_B(K)[0,u_1],\ x_2+\D_B(K)[0,u_2]\subset K$
(note that each of these segments is contained in $K^w_1$ \emph{or} $K^w_2$).

Now we will use an argument from the proof of \cite[Th.~1.4]{GMS}. Since $K$ is convex, then
$C:=\conv(\{x_1+\D_B(K)[0,u_1],\ x_2+\D_B(K)[0,u_2]\})$ is contained in $K$, and so $\A(C)\leq\A(K)$.
In this situation, a result by Groemer \cite{Gro} (see~\cite[Th.~2]{BH}) states that $\A(C)$
is minimal if both segments have a common endpoint, and thus, straightforward computations give
\[
\begin{split}
\A(K) & \geq \A(C) \geq \A(\conv(\{\D_B(K)[0,u_1],\D_B(K)[0,u_2]\})) \\
      & = \frac{\D_B(K)^2}{2}\sqrt{1-(u_1^Tu_2)^2} \geq \frac{\sqrt{3}}{4}\,\D_B(K)^2,
\end{split}
\]
which completes the proof.
\end{proof}

\subsection{The centrally symmetric case}

As in~\cite{Beh}, we will also focus on the centrally symmetric case
(considering always the origin as center of symmetry),
pursuing an isodiametric-type inequality for bisections in this setting.
% MPS lo hacen todo para bisecciones por subconjuntos con igual área. Para enunciar el siguiente resultado, enunciado para bisecciones en general, hay que remitirse a mejorabisecciones, que es donde se prueba que para bisecciones generales, hay siempre una minimizante dada por una recta que pasa por el centro de simetria (y entonces los subconjuntos tienen areas iguales). En cualquier caso, la misma prueba de MPS - enunciada en ambiente de areas iguales, pero no explícitamente escrito en el enunciado de su Prop. 4, sirve para el caso general.
The following result was proven in \cite[Prop.~3.1]{mejora} (cf.~\cite[Prop.~4]{MPS}).

\begin{lemma}%(\cite[Prop.~4]{MPS})
\label{lem:symmetrics}
Let $K\in\K^2$ be centrally symmetric.
Then, there exists a minimizing bisection $(K_1,K_2)$ of $K$ such that
$K_1\cap K_2=[-p,p]$, for some $p\in\mathrm{bd}(K)$. Consequently, $K_1=-K_2$.
%\begin{itemize}
%\item $K_1\cap K_2=[-p,p]$, for some $p\in\mathrm{bd}(K)$,
%\item $K_1=-K_2$.
%\end{itemize}
\end{lemma}

The above Lemma~\ref{lem:symmetrics} allows to obtain a necessary condition
for a given centrally symmetric convex body to be in Behrend-bisecting position.

\begin{lemma}\label{lem:Behrend-bisect-0symm}
Let $K\in\K^2$ be centrally symmetric and in Behrend-bisecting position.
For every $w\in B_K$ with $(K_1^w,-K_1^w)$ as the corresponding minimizing bisection of $K$,
we have that $K_1^w$ and $-K_1^w$ are in Behrend position.
\end{lemma}

\begin{proof}
Since $K$ is in Behrend-bisecting position and $w\in B_K$, Lemma \ref{lem:behrendBisectPos}\,(ii) implies that
for every $u\in\mathbb S^1$, there exists $v\in V_{K_1^w}\cup V_{-K_1^w}=V_{K_1^w}=V_{-K_1^w}$, such that $|u^Tv|\leq 1/\sqrt{2}$.
By Proposition \ref{prop:behrendPos}, we obtain that $K_1^w$ is in Behrend position, as well as $-K_1^w$.
\end{proof}

We can now prove Theorem~\ref{th:RevIsodiamBisect0symmIneq}, which establishes an isodiametric inequality for bisections in the
centrally symmetric case.

\begin{proof}[Proof of Theorem \ref{th:RevIsodiamBisect0symmIneq}]
Let $(K_1,K_2)$ be a minimizing bisection of $K$. We can suppose by Lemma \ref{lem:symmetrics} that $K_2=-K_1$.
As $K$ is centrally symmetric and in Behrend-bisecting position, Lemma \ref{lem:Behrend-bisect-0symm}
yields that $K_1$ (and also $K_2=-K_1$) is in Behrend position.
Thus \eqref{eq:BehrendIneq} implies that
\begin{equation}
\label{eq:twoequilateral}
\frac{\A(K)}{\D_B(K)^2}=\frac{\A(K_1)+\A(K_2)}{\D_B(K)^2}=\frac{\A(K_1)}{\D(K_1)^2}+\frac{\A(K_2)}{\D(K_2)^2}\geq
\frac{\sqrt{3}}{4}+\frac{\sqrt{3}}{4}=\frac{\sqrt{3}}{2}.
\end{equation}
\end{proof}

We will now proceed as in  % Remark~\ref{re:triangles} and
Proposition~\ref{prop:triangles},
but focusing on the affine class of the square, i.e., the parallelograms, which are centrally symmetric.
Proposition~\ref{prop:squares} shows that the \emph{only} parallelogram
in Behrend-bisecting position is the rectangle $[-1,1]\times[-2,2]$
(up to dilations and rigid motions, see Remark~\ref{re:inv}).

\begin{proposition}\label{prop:ParallelogramOptimal}
\label{prop:squares}
The unique parallelogram in Behrend-bisecting position is the rectangle $[-1,1]\times[-2,2]$.
\end{proposition}

\begin{proof}
Let $K\subset\mathbb R^2$ be a parallelogram,
and let $[-p,p]$ be a line segment determining a minimizing bisection $(K_1,K_2)$ of $K$, for some $p\in\mathrm{bd}(K)$.
Since $K$ is in Behrend-bisecting position, then $K_1$ (and $K_2=-K_1$)
is in Behrend position, by Lemma~\ref{lem:Behrend-bisect-0symm}.
%We will now search for the parallelograms in Behrend-bisecting position, i.e., the ones which maximize the isodiametric quotient.
We will distinguish two possibilities:

If $p$ is a vertex of $K$, then $K_1$ and $K_2$ are triangles.
Since the only triangle in Behrend position is the equilateral one~\cite{Beh},
then the only candidate in this case is the parallelogram $P$ formed by two congruent equilateral triangles joined by a common
edge, with isodiametric quotient $\A(P)/\D_B(P)^2=\sqrt{3}/2$, in view of~\eqref{eq:twoequilateral}.

If $p$ is not a vertex of $K$,  %a point in the relative interior of an edge of $K$}.
then $K_1$ %(and also $K_2=-K_1$)
is a quadrangle in Behrend position with two parallel edges.
We can assume that $K_1=\conv(\{p_1,p_2,p_3,p_4\})$, where $p_i\in\mathbb R^2$, $i=1,\ldots,4$.
Proposition \ref{prop:behrendPos} implies that there exist at least two different vectors $v_1,v_2\in V_{K_1}$,
and so $K_1$ contains at least two different diametrical segments.
Since $K_1$ is a quadrangle with two parallel edges,
then necessarily one of the diagonals of $K_1$, namely $[p_1,p_3]$, is a diametrical segment.
Denote by $h_1$ (resp., $h_2$) the distance from $p_2$ (resp., $p_4$) to $[p_1,p_3]$.
Then $h_1+h_2\leq d(p_2,p_4)\leq\D(K_1)$, and so
%If $h_1$ (resp.~$h_2$) is the distance from $p_2$ (resp.~$p_4$) to $[p_1,p_3]$, then
\[
\A(K_1)=\frac{1}{2}\,\D(K_1)\,(h_1+h_2)\leq\frac{\D(K_1)^2}{2}.
\]
Since $K_2=-K_1$, we will also have that $\A(K_2)\leq\D(K_2)^2/2$.
Then,
\[
\frac{\A(K)}{\D_B(K)^2}=\frac{\A(K_1)+\A(K_2)}{\D_B(K)^2}=\frac{\A(K_1)}{\D(K_1)^2}+\frac{\A(K_2)}{\D(K_2)^2}\leq\frac{1}{2}+\frac12=1.
\]
Moreover, we have equality above if and only if $h_1+h_2=\D(K_1)$.
This is equivalent to the fact that $[p_2,p_4]$ is orthogonal to $[p_1,p_3]$,
i.e., when $K_1$ (and thus $K_2$) is a square.
This implies that $K=K_1\cup K_2$ is a rectangle of the form $[-1,1]\times[-2,2]$.
Since this rectangle has isodiametric quotient greater than or equal to
the isodiametric quotient of $P$, the statement holds.
% Finally, since the isodiametric quotient of the parallelogram $P$ (consisting of two joined equilateral triangles) is equal to $\sqrt{3}/2$, whereas the corresponding one for the rectangle $[-1,1]\times[-2,2]$ equals $1$, we conclude that the only parallelogram in Behrend-bisecting position is that rectangle.
\end{proof}

\begin{remark}
\label{re:squares}
A remarkable consequence from Proposition~\ref{prop:squares} is that
the \emph{necessary} condition in Lemma~\ref{lem:Behrend-bisect-0symm} is \emph{not sufficient}
(analogously to Remark~\ref{re:triangles}):
the parallelogram consisting of two equilateral triangles touching
in a common edge, both of them in Behrend position~\cite{Beh},
\emph{is not} in Behrend-bisecting position.
\end{remark}

The previous Proposition~\ref{prop:squares} suggests that the inequality
from our Theorem~\ref{th:RevIsodiamBisect0symmIneq} is not sharp,
leadings us to the following conjecture.

\begin{conjecture}
Let $K\in\K^2$ be centrally symmetric and in Behrend-bisecting position. Then,
\[
\frac{\A(K)}{\D_B(K)^2}\geq 1,
\]
with equality if and only if $K=[-1,1]\times[-2,2]$.
\end{conjecture}

\section{The isominwidth-bisecting position and the reverse isominwidth inequality}
\label{sec:RevIsominwidthIneq}

In this section we will establish a reverse isominwidth inequality,
following the same scheme as in Section~\ref{sec:RevIsodiamIneq}.
In order to obtain such an inequality, we will focus on the planar convex bodies in isominwidth-bisecting position,
defined by equality~\eqref{eq:isominwidthoptimalposition}.
Our first observation is that the infimum in \eqref{eq:isominwidthoptimalposition}
is actually a minimum, and so, for any given $K\in\K^2$ there exists an affine representative in
isominwidth-bisecting position (we will omit the proof of this fact since it is completely analogous to Lemma~\ref{lem:isodiamMaximum}).
Notice also that $\w_B(K)=\w(K)/2$ by Lemma~\ref{lem:WidthBisect=Half}, and so
\begin{equation}
\label{eq:4}
\min_{\phi\in End(\mathbb R^{2})}\frac{\A(\phi(K))}{\w_B(\phi(K))^2}
=4\,\min_{\phi\in End(\mathbb R^{2})}\frac{\A(\phi(K))}{\w(\phi(K))^2}.
\end{equation}
This equality immediately gives the following Corollary~\ref{cor:isominwidthEquiv},
which states a new equivalence for the planar convex bodies in isominwidth optimal position,
defined by~\eqref{eq:iop} and introduced in ~\cite{GMS} (see~\cite[Th.~5.3]{GMS} for some other related equivalences).

\begin{corollary}\label{cor:isominwidthEquiv}
Let $K\in\K^2$. The following statements are equivalent:
\begin{enumerate}
\item[(i)] $K$ is in isominwidth-bisecting position.
\item[(ii)] $K$ is in isominwidth optimal position.
\end{enumerate}
\end{corollary}

Finally, we can prove Theorem \ref{th:RevIsominwidthBisectIneq}.

\begin{proof}[Proof of Theorem \ref{th:RevIsominwidthBisectIneq}]
By Corollary \ref{cor:isominwidthEquiv}, $K$ is in isominwidth optimal position,
and taking into account Lemma~\ref{lem:WidthBisect=Half} and~\eqref{eq:GMSchymuraIneq} we conclude that
\[
\frac{\A(K)}{\w_B(K)^2}=4\,\frac{\A(K)}{\w(K)^2}\leq 4.
\]
The equality case follows directly from the corresponding equality case in \eqref{eq:GMSchymuraIneq}.
\end{proof}

\section{Other spaces}\label{sec:OtherSpaces}

In this section we will briefly discuss how most of the above definitions and posed problems can be extended to other spaces.
We will also point out some of the technical difficulties that we find in order to go on solving these problems in those settings.

\subsection{Isodiametric and Isominwidth bisecting inequalities in $\mathbb R^n$}
\label{subsec:rn}

Let $K\subset\mathbb R^n$ be a convex body with non-empty interior,
and denote by $\mathrm{V}(K)$ be the n-dimensional volume of $K$.
Recall that the diameter $\D(K)$ of $K$ is the maximum distance between any two points of $K$, whereas
the minimum width $\w(K)$ of $K$ is the minimum distance between two parallel hyperplanes containing $K$ between them.
We first extend the notion of bisection previously introduced for the planar setting.
Let $\mathbb B^n_2$ be the Euclidean unit ball of $\mathbb R^n$.

For a convex body $K$ in $\mathbb R^n$, a \emph{bisection} of $K$ will be any pair of closed sets $(K_1,K_2)$ satisfying that
\begin{itemize}
\item[(i)] $K=K_1\cup K_2$,
\item[(ii)] $K_1\cap K_2=l(\mathbb B^{n-1}_2)$, where $l:\mathbb B^{n-1}_2\rightarrow K$ is an injective and continuous map such that
$l(\mathbb B^{n-1}_2)\cap\mathrm{bd}(K)=l(\mathrm{bd}(\mathbb B^{n-1}_2))$.
\end{itemize}
We will denote by $\mathcal B(K)$ the set of all bisections of $K$.
We can now define the \emph{infimum of the maximum bisecting diameter} of $K$ by
\[
\D_B(K):=\inf_{(K_1, K_2)\in\mathcal{B}(K)}\max\{\D(K_1),\D(K_2)\}.
\]
A remarkable difference with respect to the planar case is that it is not clear now whether this infimum is
a minimum. The reason is that for an arbitrary bisection $(K_1,K_2)$ of $K$ in $\mathbb R^n$, for $n\geq 3$,
the set $K_1\cap K_2\cap \mathrm{bd}(K)$ is, in general, $n$-dimensional,
and so it will not induce a bisection by a hyperplane
(cf. Lemma~\ref{lem:StraightLine}).
This suggests that an appropriate approach could be focusing on bisections by hyperplanes,
which will imply that the infimum is actually a minimum, by using Blaschke Selection Theorem.
% in order to apply Blaschke Selection Theorem, we need a sequence of compact convex sets. But now the bisections could give subsets which are not convex. Recall that in the planar case, from a general bisection we can consider the corresponding one given by a line segment, which yields convex subsets. Now in higher dimensions this trick cannot be done.

We now sketch that the corresponding isodiametric quotient is upper bounded.
For a given convex body $K$ in $\mathbb R^n$,
let $x,y\in K$ be points such that $d(x,y)=\D(K)$, and
let $(K_1,K_2)\in\mathcal{B}(K)$.
Then at least two points from $\{x,(x+y)/2,y\}$ belong to one of the sets $K_1$ or $K_2$.
Since the distance between any pair of those three points is at least $\D(K)/2$,
then we can conclude that $\max\{\D(K_1),\D(K_2)\}\geq\D(K)/2$,
and so $\D_B(K)\geq \D(K)/2$,
which together with the classical isodiametric inequality in $\mathbb R^n$
(see \eqref{eq:bieberbach} for the planar case) gives
\[
\frac{\mathrm{V}(K)}{\D_B(K)^n}\leq 2^n\frac{\mathrm{V}(K)}{\D(K)^n}\leq 2^n\frac{\mathrm{V}(\mathbb B^n_2)}{\D(\mathbb B^n_2)^n},
\]
thus showing that this quotient is upper bounded by an absolute positive constant.
Hence the supremum of $V(K)/\D_B(K)^n$ over the convex bodies in $\mathbb R^n$ is finite
and it would be interesting to characterize the sets attaining such value,
as done in Theorem~\ref{th:isodiametricIneq}.

Analogously, we can define the \emph{infimum of the maximum bisecting width} of $K\subset\mathbb R^n$ by
\[
\w_B(K):=\inf_{(K_1, K_2)\in\mathcal{B}(K)}\max\{\w(K_1),\w(K_2)\}.
\]
Using analogous ideas to the ones exhibited in Lemma \ref{lem:WidthBisect=Half},
we can see that $\w(K)=2\,\w_B(K)$. Notice that this implies that the previous infimum is in fact a minimum:
if $w(K)$ is attained between two parallel supporting hyperplanes $H_1$, $H_2$,
then $\w_B(K)$ will be attained by the bisection of $K$ given by the hyperplane $(H_1+H_2)/2$.
Moreover, taking into account that
\[
\frac{\mathrm{V}(K)}{\w(K)^n}\geq\frac{2}{\sqrt{3}\,n!}  % desigualdad estricta si $n\geq 3$ (sin tener caracterizado el caso de igualdad). Para $n=2$ sí se tiene el mayor o igual, by Pal, con igualdad sii tenemos un triangulo equilatero.
\]
(see~\cite[Th.~6.2]{Be}), we can conclude that
\[
\frac{\mathrm{V}(K)}{\w_B(K)^n}=2^n\frac{\mathrm{V}(K)}{\w(K)^n}\geq\frac{2^{n+1}}{\sqrt{3}\,n!},
\]
thus showing that this quotient is lower bounded by an absolute constant.

\begin{remark}
Some interesting results regarding the isodiametric quotient
in compact convex surfaces of $\mathbb R^3$ can be found in~\cite{surfaces1,surfaces2}.
\end{remark}

\subsection{Reverse Isodiametric and Isominwidth bisecting inequalities in $\mathbb R^n$}
\label{subsec:rn2}

Using the same ideas commented in the Introduction for the planar case,
and the same definitions from Subsection~\ref{subsec:rn},
we can see that the quotient $\mathrm{V}(K)/\D_B(K)^n$ cannot be lower bounded
by any positive constant, when considering arbitrary convex bodies $K\subset\mathbb R^n$.
However, we can develop the same approach from Section~\ref{sec:RevIsodiamIneq}:
we can say that a convex body $K\subset\mathbb R^n$ is in \emph{Behrend-bisecting position} if
\[
\frac{\mathrm{V}(K)}{\D_B(K)^n}=\sup_{\phi\in End(\mathbb{R}^n)}\frac{\mathrm{V}(\phi(K))}{\D_B(\phi(K))^n}.
\]
The ideas from \cite[Lemma 3.2]{GMS} allow to obtain
a result analogous to Lemma~\ref{lem:behrendBisectPos},
which will lead to the following consequence, see~\cite[Th.~1.4]{GMS}:
if $K$ is a convex body in $\mathbb R^n$ in Behrend-bisecting position, then
\[
\frac{\mathrm{V}(K)}{\D_B(K)^n}\geq\frac{1}{\sqrt{n!}\,n^\frac{n}{2}}.
\]
As we noted in~\eqref{eq:RevIsodBisecting}, this inequality is not sharp.

%XXX one can show that if a convex compact set $K\in\mathbb R^n$ is in Behrend-bisecting position, then,
%for every $L$ i-dimensional linear subspace, every $u\in\mathbb S^{n-1}\cap L$, and every minimizing bisection $(K_1, K_2)$ of $K$,
%we have that
%\begin{enumerate}
%\item[(i)] there exists $v\in V_{K_1}\cup V_{K_2}$ such that $|u^Tv|\geq \sqrt{i/n}$, and
%\item[(ii)] there exists $v\in V_{K_1}\cup V_{K_2}$ such that $|u^Tv|\leq \sqrt{i/n}$.
%\end{enumerate}
%
%From here one can directly use the same proof as in \cite[Theorem 1.4]{GMS} and show that if
%$K\in\mathcal K^n$ is in Behrend bisecting position then
%\[
%\frac{\mathrm{V}(K)}{\D_B(K)^n}\geq\frac{1}{\sqrt{n!}\,n^\frac{n}{2}},
%\]
%and this inequality is not tight.
%XXX

Analogously, the quotient $\mathrm{V}(K)/\w_B(K)^n$ cannot be upper bounded by any positive constant
(just consider a \emph{very flat} convex body $K$ in $\mathbb R^n$).
However, if we assume that $K$ is in \emph{isominwidth-bisecting position}, i.e., if
\[
\frac{\mathrm{V}(K)}{\w_B(K)^n}=\inf_{\phi\in End(\mathbb{R}^n)}\frac{\mathrm{V}(\phi(K))}{\w_B(\phi(K))^n},
\]
then one could prove that this quotient is upper bounded. More precisely,
\[
\frac{\mathrm{V}(K)}{\w_B(K)^n}\leq 2^n,
\]
with equality if and only if $K$ is a cube
(this follows from~\cite[Th.~1.6]{GMS}, which is an extension of~\eqref{eq:GMSchymuraIneq} to higher dimensions,
together with $\w_B(K)=\w(K)/2$, as in Lemma~\ref{lem:WidthBisect=Half},
and the fact that $K$ is in isominwidth position if and only if $K$ is in isominwidth-bisecting position,
as in Corollary~\ref{cor:isominwidthEquiv}).

\subsection{Isodiametric inequality in the Spherical and Hyperbolic space}
\label{subsec:hn}
The study of general geometric inequalities can be also done in other geometries, different from the Euclidean one.
In this direction, some interesting results have been obtained in
the spherical space $\mathbb S^n$ and the hyperbolic space $\mathbb H^n$ of dimension $n$~\cite{dekster, klain, HCMF}.
In this general context, a set $K$ is called \emph{convex} if for any $x,y\in K$, \emph{the shortest geodesic segment} joining
$x$ and $y$ is contained in $K$ (in the spherical case, it is additionally required that $K$ is contained in a halfsphere).
Moreover, one can naturally define the notions of \emph{spherical diameter}, \emph{spherical width}
and the \emph{spherical area}, as well as the corresponding \emph{hyperbolic} analogues.
In this setting, the isodiametric and isominwidth inequalities have been recently proven in the 2-dimensional spherical and hyperbolic cases,
when $K$ is \emph{centrally-symmetric}~\cite[Th.~1.1 and~1.3, Th.~5.1 and~5.3]{HCMF}.
Some other related considerations in the hyperbolic case can be found in \cite{GS}.

We would like to note that, in $\mathbb S^n$ and $\mathbb H^n$,
we cannot assure that the diameter of a given set $A$ \emph{is attained}
by a pair of its extreme points (they can be defined as the points of $A$ which do not belong to the relative interior
of any geodesic segment contained in $A$), as it occurs in $\mathbb R^n$. For instance, consider
\[
A_{\delta}=\mathrm{conv}_{\mathbb S^2}\left(\{(-1,0,0)\}\cup\left\{\sin\theta(0,1,0)+\cos\theta(\delta,0,\sqrt{1-\delta^2}):\theta\in[-\pi/4,\pi/4]\right\}\right)\subset\mathbb S^2,
\]
for some $\delta\in[1/4,3/4]$
(notice that $A_\delta$ is just a geodesic triangle in $\mathbb S^2$).
%with vertices ).
%Notice that if $\A\subset\mathbb S^{n}$, then $\mathrm{conv}_{\mathbb S^{n}}(A)$ is the \emph{convex hull of $A$ in $\mathbb S^{n}$}, i.e., the set of points which are given as a convex combination of points of $A$ by maximum arcs of $\mathbb S^{n}$.
Then, the diameter of $A_{\delta}$ in $\mathbb S^2$ is only given by the distance between
the points $(-1,0,0)$, $(\delta,0,\sqrt{1-\delta^2})\in A_{\delta}$, but it is clear that
$(\delta,0,\sqrt{1-\delta^2})$ is not an extreme point of $A_{\delta}$.

Of course, for a given set $A$ contained in $\mathbb S^n$ or in $\mathbb H^n$
we can consider the problems studied in Sections~\ref{sec:IsodiamIneq} and~\ref{sec:IsominwidthIneq}.
From the previous example $A_\delta\subset\mathbb S^2$,
it is not clear that an analogous result to Lemma~\ref{lem:StraightLine} can be obtained in this setting.
Notice that if we substitute a general bisection of a given set $A\subset\mathbb S^2$
by the bisection determined by the maximum arc with the same endpoints,
the corresponding diameters of the new subsets can be greater than the former ones,
since they are not necessarily attained by extreme points of the subsets.
Furthermore, and up to our knowledge, there is no isodiametric inequality in the literature
for \emph{general} convex bodies in $\mathbb S^n$ or $\mathbb H^n$,
which suggests that these problems will require a more detailed study.

%\emph{Acknowledgements. } We would like to thank the anonymous referee for helping us in improving
%the write-up of the manuscript, which also pushed us to consider the problems beyond the planar case (in Section \ref{sec:OtherSpaces}).

\end{document}